\documentclass[10pt, a4paper]{amsart}

\usepackage[sc]{mathpazo}

\usepackage[utf8]{inputenc}
\usepackage[T1]{fontenc}
\usepackage[UKenglish]{babel}

\usepackage{float}
\usepackage{tikz}
\usepackage{mathrsfs}
\usepackage{lscape}
\usepackage[all,knot]{xy}
\usepackage{enumerate}

\usepackage{amsmath}
\usepackage{amsfonts}
\usepackage{amssymb}
\usepackage{amsthm}

\newcommand{\osplit}[2]{\begin{scope}[shift={(#1+2,#2+2)}]
\draw (-2,-2) .. controls +(1,1) and +(-1,1) ..  (2,-2);
\draw (-2,+2) .. controls +(1,-1) and +(-1,-1) ..  (2,+2);
\end{scope}}

\newcommand{\isplit}[2]{\begin{scope}[shift={(#1+2,#2+2)}]
\draw (-2,-2) .. controls +(1,1) and +(1,-1) ..  (-2,+2);
\draw  (2,-2).. controls +(-1,1) and +(-1,-1) ..  (2,+2);
\end{scope}}

\newcommand{\ocross}[2]{\begin{scope}[shift={(#1+2,#2+2)}]
\draw (+2,-2) -- (-2,+2);
\pgfsetlinewidth{8*\pgflinewidth}
\draw[white] (-2,-2) -- (+2,+2);
\pgfsetlinewidth{.125*\pgflinewidth}
\draw (-2,-2) -- (+2,+2);
\end{scope}}

\newcommand{\icross}[2]{\begin{scope}[shift={(#1+2,#2+2)}]
\draw (-2,-2) -- (+2,+2);
\pgfsetlinewidth{8*\pgflinewidth}
\draw[white] (+2,-2) -- (-2,+2);
\pgfsetlinewidth{.125*\pgflinewidth}
\draw (+2,-2) -- (-2,+2);
\end{scope}}

\newcommand{\ocrosstext}{\raisebox{-0.1cm}{\begin{tikzpicture}[scale=.07]
\begin{scope}[shift={(2,2)}]
\draw (+2,-2) -- (-2,+2);
\pgfsetlinewidth{8*\pgflinewidth}
\draw[white] (-2,-2) -- (+2,+2);
\pgfsetlinewidth{.125*\pgflinewidth}
\draw (-2,-2) -- (+2,+2);
\end{scope}
\draw[dashed] (2,2) circle (2.8cm);
\end{tikzpicture}}\:}

\newcommand{\isplittext}{\raisebox{-0.1cm}{\begin{tikzpicture}[scale=.07]
\begin{scope}[shift={(2,2)}]
\draw (-2,-2) .. controls +(1,1) and +(1,-1) ..  (-2,+2);
\draw  (2,-2).. controls +(-1,1) and +(-1,-1) ..  (2,+2);
\end{scope}
\draw[dashed] (2,2) circle (2.8cm);
\end{tikzpicture}}\:}

\newcommand{\osplittext}{\raisebox{-0.1cm}{\begin{tikzpicture}[scale=.07]
\begin{scope}[shift={(2,2)}]
\draw (-2,-2) .. controls +(1,1) and +(-1,1) ..  (2,-2);
\draw (-2,+2) .. controls +(1,-1) and +(-1,-1) ..  (2,+2);
\end{scope}
\draw[dashed] (2,2) circle (2.8cm);
\end{tikzpicture}}\:}

\theoremstyle{plain}
	\newtheorem{theorem}{\noindent\textbf{Theorem}}[section]
	\newtheorem*{theorem*}{\noindent\textbf{Theorem}}
	\newtheorem{lemma}[theorem]{\noindent\textbf{Lemma}}
	\newtheorem{proposition}[theorem]{\noindent\textbf{Proposition}}
	\newtheorem{corollary}[theorem]{\noindent\textbf{Corollary}}

\theoremstyle{definition}
	\newtheorem*{definition*}{\noindent\textbf{Definition}}

\theoremstyle{remark}
	\newtheorem{remark}{\noindent\textbf{Remark}}[section]
	
\theoremstyle{definition}
	\newtheorem{definition}{\noindent\textbf{Definition}}[section]

\begin{document}
\title{On transverse invariants from Khovanov-type homologies}
\author{Carlo Collari}

\maketitle
\begin{abstract}
In this article we introduce a family of transverse invariants arising from the deformations of Khovanov homology. This family includes the invariants introduced by Plamenevskaya and by Lipshitz, Ng, and Sarkar. 
Then, we investigate the invariants arising from Bar-Natan's deformation. These invariants, called $\beta$-invariants, are essentially equivalent to Lipshitz, Ng, and Sarkar's invariants $\psi^\pm$. From the $\beta$-invariants we extract two non-negative integers which are transverse invariants (the $c$-invariants). Finally, we give several conditions which imply the non-effectiveness of the $c$-invariants, and use them to prove several vanishing criteria for the Plamenevskaya invariant $[\psi]$, and the non-effectiveness of the vanishing of $[\psi]$, for all prime knots with less than 12 crossings.
\end{abstract}

\section{Introduction}
\subsection*{Motivations and background}


A \emph{contact $3$-manifold} is a $3$-dimensional (smooth) manifold $\mathcal{M}$ endowed with a totally non-integrable plane distribution $\xi$. A \emph{link} in $\mathcal{M}$ is a smooth embedding of a number of copies of $\mathbb{S}^1$ into $\mathcal{M}$. A link is called \emph{transverse} with respect to $\xi$ if it is nowhere tangent to $\xi$. Two transverse links in $(\mathcal{M},\xi)$ are \emph{equivalent} if they are ambient isotopic through a one-parameter family of transverse links. Transverse links are a central object of study in low-dimensional contact topology. In this paper we are concerned with the study of transverse links in a special setting. Fix a system of coordinates $(x,y,z)$ on $\mathbb{R}^3$. The \emph{symmetric contact structure} is the plane field $\xi_{sym} = ker(dz + xdy - ydx)$.
Since we are interested only in transverse links in $(\mathbb{R}^3,\xi_{sym})$, henceforth we shall restrict ourselves to this setting. 
Once an orientation of $\mathbb{R}^3$ is fixed, each transverse link inherits a natural orientation from the contact structure (\cite{Etnyre05}). All transverse invariants defined in this paper are defined with respect to this orientation.

\begin{remark}
With an abuse of language, the words ``link'',``knot'',``transverse link'' and ``transverse knot'' shall denote both embeddings and equivalence classes of embeddings.
\end{remark}

A natural way to tackle the problem of classifying transverse links is to produce invariants capable of distinguishing them. It follows immediately from the definitions that the underlying link (i.e. the ambient isotopy class of the oriented link) is a transverse invariant. Another well-known invariant for transverse links is the \emph{self-linking number}, which is denoted by $sl$. The underlying link and the self-linking number\footnote{For transverse links there is a refinement of the self-linking number which is called the self-linking matrix. Strictly speaking, the classical invariants for transverse links are the underlying link and the self-linking matrix. However, for consticency with the literature (\cite{Nglipsar13, Plamenevskaya06}) we are going to use the self-linking number instead of the self-linking matrix.} are called \emph{classical invariants}. 
These are by no means complete invariants for transverse links. Nonetheless, there are links, which are called \emph{simple}, whose transverse representatives are classified by their self-linking number (e.g. unknot, the positive torus knots). Among the earliest examples of non-simple links there are those given by Etnyre and Honda (\cite{EtnyreHonda03}), and by Birman and Menasco (\cite{BirmanMenasco06II}). More precisely, Birman and Menasco described a construction which yelds pairs of non-equivalent transverse links with the same classical invariants: the \emph{negative flype}\footnote{For coerence, we use the word ``flype'' with the same meaning as in \cite{Nglipsar13}.}. We call a pair of (non-equivalent) transverse links with the same classical invariants a \emph{non-simple pair}. Any transverse invariant capable of distinguishing the elements of a non-simple pair is called \emph{effective}.
\begin{remark}
There is a paucity of non-simple pairs, whose crossing number is sufficiently small, which are not obtained via negative flype (cf. \cite[Section 1 and Subsection 4.5]{Nglipsar13}). Therefore, there is a lack of sufficiently simple examples where to test the effectiveness of transverse invariants which are not capable of distinguishing negative flypes.
\end{remark}

In classical link theory, the well-know theorems of Alexander and Markov assert, respectively, that any oriented link in $\mathbb{R}^3$ can be represented as the closure of a braid, and that all these representations are related by a finite sequence of combinatorial moves (called Markov moves). In the theory of transverse links there are similar results; the first, due to Bennequin (\cite{Bennequin83}) states that  all transverse links can be represented as closed braids. The second result, due to Orevkov and Shevchishin (\cite{OrevkovShev03}) and, independently, to Wrinkle (\cite{Wrinkle03}), provides a complete set of combinatorial moves relating all braids whose closure represent the same transverse link. We summarise these results into the following theorem, which shall be referred as the \emph{transverse Markov theorem} in the rest of the paper.

\begin{theorem}[Bennequin \cite{Bennequin83}, Orevkov and Shevchishin \cite{OrevkovShev03}, Wrinkle  \cite{Wrinkle03}]
Any transverse link is transversely isotopic to the closure of a braid (with axis the $z$-axis).
Moreover, two braids represent the same transverse link if and only if they are related by a finite sequence of braid relations, conjugations, positive stabilizations, and positive destabilizations\footnote{Let $B\in B_{m-1}$, the \emph{positive} (resp. \emph{negative}) \emph{stabilization} of $B\in B_{m-1}$ is the braid $B\sigma_{m}\in B_{m}$ (resp. $B\sigma_{m}^{-1}\in B_{m}$). The destabilization is just the inverse process: if one considers a braid of the form $A\sigma_{m} B$ (resp. $A\sigma_{m}^{- 1} B$), where $A,\ B\in B_{m-1}$, then its positive (resp. negative) destabilization is the braid $AB$.}. These moves are called \emph{transverse Markov moves}.
\end{theorem}

\begin{remark}
Braids are naturally oriented, and their orientation coincide with the orientation of the corresponding transverse link. 
\end{remark}

\begin{remark}
By adding the negative stabilization to the set of transverse Markov moves one recovers the full set of \emph{Markov moves}.
\end{remark}

Any sequence of Markov moves between braids, naturally translates into a sequence of oriented Reidemeister moves between their closures. In particular, conjugation in the braid group can be seen as a sequence of Reidemeister moves of the second type followed by a planar isotopy, while the braid relations can be seen as either second or third Reidemeister moves. We remark that not all the oriented versions of second and third Reidemeister moves arise in this way; those which can be obtained as composition of Markov moves and braid relations are called \emph{braid-like} or \emph{coherent}. Finally, a positive (resp. negative) stabilization translates into a positive (resp. negative) first Reidemeister move, as shown in Figure \ref{fig:stabilizationasr1}. 

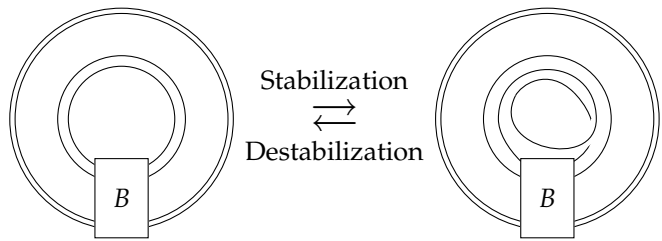
\begin{figure}[h]
\centering
\begin{tikzpicture}[scale =.7]

\begin{scope}[shift = {+(-8,0)}]
\draw (0,2) circle (1.2);
\draw (0,2) circle (1);
\draw (0,2) circle (2);
\draw (0,2) circle (2.1);
\draw[fill, white] (-.5,-.25) rectangle (0.5,1.25);
\draw (-.5,-.25) rectangle (0.5,1.25);
\node (a) at (0,.5) {$B$};
\end{scope}

\node at (-4,2.75) {Stabilization};
\node at (-4,2) {\huge{$\rightleftarrows$}};
\node at (-4,1.4) {Destabilization};
\draw (0,2) circle (1.2);
\draw (0,2) circle (2);
\draw (0,2) circle (2.1);
\draw[fill, white] (-.5,-.25) rectangle (0.5,1.25);
\draw (-.5,-.25) rectangle (0.5,1.25);
\node (a) at (0,.5) {$B$};

\draw (0.5,1.2) .. controls +(.75,.5) and +(.1,-.1) .. (.5,2.5);
\draw  (.5,2.5) .. controls +(-.5,.5) and +(.1,.2) .. (-0.6,2.44);
\pgfsetlinewidth{10*\pgflinewidth}
\draw[white]  (.92,2) .. controls +(0,-1) and +(-.5,-1) .. (-0.6,2.44);
\pgfsetlinewidth{.1*\pgflinewidth}
\draw  (.92,2) .. controls +(0,-1) and +(-.5,-1) .. (-0.6,2.44);

\draw  (.92,2) .. controls +(0,1.25) and +(0,1.25) .. (-.92,2) .. controls +(0,-.25) and +(-.25,.25) .. (-.5,1.2);

\end{tikzpicture}
\caption{Stabilization and destabilization as Reidemeister moves between closures.}\label{fig:stabilizationasr1}
\end{figure}

Let $B$ be a braid. The \emph{self-linking number} of $B$ is defined as follows:
\[ sl(B) = w(B) - b(B), \]
where $w$ denotes the \emph{writhe}, and $b$ denotes the \emph{braid index} (i.e. the number of strands).
In the light of the transverse Markov theorem, it is immediate that the self-linking number is a transverse invariant. (The integer $sl(B)$ is indeed the self-linking number of the transverse link represented by $B$.) Moreover, from the definition of $sl$ follows easily that a negative stabilization (resp. destabilization) does not preserve the transverse link.

Now, let us turn to the known results concerning transverse invariants in Khovanov-type homologies.
By \emph{Khovanov-type homologies} we mean all link homology theories obtained from a given Frobenius algebra using the construction originally due to Khovanov (\cite{BarNatan05cob, Khovanov00}). This construction shall be reviewed more in detail in Section \ref{Sec:Khovanovhom}. For the moment the reader should keep in mind that, given a Frobenius algebra $\mathcal{F}$ (over a ring $R$), there is a way to associate to each oriented link diagram $D$ a chain complex $C^{\bullet}_{\mathcal{F}}(D,R)$.  Moreover, this complex is combinatorially defined and can be endowed with either a second grading or a filtration. Furthermore, it is possible to associate to each sequence of Reidemeister moves between two diagrams a chain homotopy equivalence between the corresponding complexes. 

To each braid we associate the chain complex corresponding to its closure. Similarly, the map associated to a sequence of Markov moves between two braids is the map associated to the corresponding sequence of Reidemeister moves between the braid closures.

The first invariant for transverse links in a Khovanov-type homology is due to Plamenevskaya (\cite{Plamenevskaya06}). Given a braid $B$, its \emph{Plamenevskaya invariant} is a homology class $[\psi(B)]$, of bi-degree $(0,sl(B))$, in the  Khovanov homology of $\widehat{B}$. This homology class is invariant in the following sense; given two braids $B$ and $B^\prime$ related by a sequence of transverse Markov moves $\Sigma$, we have
\[(\Phi_{\Sigma})_{*}[\psi(B)] = [\psi(B^\prime)],\]
where $\Phi_{\Sigma}$ denotes the map associated to $\Sigma$.
After Plamenevskaya's groundbreaking work, invariants of similar flavour have been introduced in Heegard-Floer homologies. For example, the Ng-Oszv\'{a}th-Thurston $\widehat{\theta}$-invariant in $\widehat{HFK}$ (\cite{NgOzsvathThurston08}). The vanishing of this invariant is an effective invariant; there exist a non-simple pair of transverse knots, say $ T$ and $ T^\prime$, such that:
\[ \widehat{\theta}(T) = 0 \ne \widehat{\theta}(T^\prime).\]
Since the $\widehat{\theta}$-invariant can be interpreted as the Heegaard-Floer analogue of the Plamenevskaya invariant, it is natural to ask whether or not the vanishing of $[\psi]$ is effective. This is an open question at the time of writing. 

More recently, Lipshitz, Ng and Sarkar introduced (in \cite{Nglipsar13}) two invariants in the Khovanov-type homology associated to the Frobenius algebra $TLee$ (see Subsection \ref{Sec:exFA}). The  chain complex associated to $TLee$ is naturally filtered, denote by $\mathscr{F}_{i}C_{TLee}^{\bullet}$ this filtration. The \emph{Lipshitz-Ng-Sarakar} (\emph{LNS}) \emph{invariants} are two elements $\psi^{+}(B)$ and $\psi^{-}(B)$ in $\mathscr{F}_{sl(B)}C_{TLee}^{0}(\widehat{B},R)$, and are invariant in a sense which is made precise in the following proposition.

\begin{proposition}[Theorem 4.2 and Proposition 4.7 of \cite{Nglipsar13}]\label{prop:LipshitzNgSarkar}
Let $B$ and $B^\prime$ be two braids. If $\Phi$ is the map induced by a sequence of transverse Markov moves from $B$ to $B^\prime$, then
\[ \Phi(\psi^{*}(B)) = \pm \psi^{*}(B^\prime) + d_{TLee} \theta,\]
where $\theta \in \mathscr{F}_{sl(B)}C^{-1}_{TLee}(\widehat{B^\prime},R)$ and $\ast \in \{ + ,-\}$. Moreover, if $\Phi$ is the map induced by a negative stabilization (resp. destabilization), then
\[ \Phi(\psi^{*}(B)) = \pm \psi^{*}(B^\prime) + d_{TLee} \theta^\prime\: \in \mathscr{F}_{sl(B^\prime)}C^{-1}_{TLee}(\widehat{B^\prime},R)\quad \text{( resp. }\mathscr{F}_{sl(B)}C^{-1}_{TLee}(\widehat{B^\prime},R)\text{)},\]
where $\theta^\prime \in \mathscr{F}_{sl(B^\prime)}C^{-1}_{TLee}(\widehat{B^\prime},R)$ (resp. $\theta^\prime \in \mathscr{F}_{sl(B)}C^{-1}_{TLee}(\widehat{B^\prime},R)$) and $\ast \in \{ + ,-\}$.
\end{proposition}

\begin{remark}
The maps associated to the Reidemeister moves (and also the set of Reidemiester moves) used by Lishitz, Ng and Sarkar, are different from the ones used in the present paper.
\end{remark}

As the reader may notice, the filtration plays a key role. Lipshitz, Ng and Sarkar defined also a family of auxiliary invariants $\psi^{\pm}_{p,q}$, for $p\leq 0 < q$, as the images of $\psi^{\pm}$ under the composition
\[\mathscr{F}_{sl(B)}C_{TLee}^{0}(\widehat{B},R) \hookrightarrow \mathscr{F}_{sl(B) + 2p}C_{TLee}^{0}(\widehat{B},R) \twoheadrightarrow \frac{\mathscr{F}_{sl(B)+ 2p}C_{TLee}^{0}(\widehat{B},R)}{\mathscr{F}_{sl(B) +2 q}C_{TLee}^{0}(\widehat{B},R)}\]
In particular, one obtains that $\psi^{+}_{0,1}(B) = \psi^{-}_{0,1}(B)$ can be identified with $\psi(B)$. As a consequence, the non-effectiveness of the Lipshitz-Ng-Sarkar invariants implies the non-effectiveness of (the homology class of) the Plamenevskaya invariant.
To this end, it has been proved by Lipshitz, Ng and Sarkar (\cite{Nglipsar13}) that $\psi^\pm$ cannot distinguish non-simple pairs obtained via negative flypes, and all non simple pairs whose underlying knot is a prime knot with crossing number $< 12$ with the exception of 29 cases (see \cite[Subsection 4.5]{Nglipsar13}).  In particular, the Plamenevskaya invariant (and its vanishing) is non-effective for the transverse links described above. Moreover,  Lipshitz, Ng and Sarkar proved that the Plamenvskaya invariant cannot distinguish transverse links which become transversely isotopic after a single negative stabilization. However, it is still unknown (to the best of the author's knowledge) whether there are non-simple pairs representing the remaining 29 knots with less than $12$ crossing, and if the Plamenevskaya invariant, its vanishing or the LNS invariants can be used to distinguish them.

\subsection*{Statement of results}

The aim of this paper is to investigate the transverse information contained in Khovanov-type homologies. We start by providing a machinery to produce transverse invariants in Khovanov-type homologies, generalising the Plamenevskaya and LNS invariants. One of the interesting features of our construction is that it can be generalised to equivariant Khovanov-Rozansky homologies (\cite{Krasner09}). This will be the subject of forthcoming work.

Going back to the present paper, in Section \ref{sec:invariancebeta} we show how to associate to each \emph{oriented link diagram} $D$, and each Frobenius algebra $\mathcal{F}$ (over the ring $R$) as in Subsection \ref{Sec:exFA}, two cycles denoted by $\beta_{\mathcal{F}}(D,R)$ and $\overline{\beta}_{\mathcal{F}}(D,R)$. One of the main results in this paper is the following

\begin{theorem}\label{theorem:main1}
Let $\mathcal{F}$ be a Frobenius algebra as in Subsection \ref{Sec:exFA}. Given a braid $B$, there exist two (possibly equal) cycles $ \beta_{\mathcal{F}}(B) = \beta_{\mathcal{F}}(\widehat{B},R)$, $ \overline{\beta}_{\mathcal{F}}(B) = \overline{\beta}_{\mathcal{F}}(\widehat{B},R)   \in C_{\mathcal{F}}^{0}(\widehat{B},R)$ such that: if the braids $B$ and $B^\prime$ are related by a sequence $\Sigma$ of \emph{transverse} Markov moves then the map
\[  \Phi_{\Sigma}: C_{\mathcal{F}}^{\bullet}(\widehat{B},R) \longrightarrow C_{\mathcal{F}}^{\bullet}(\widehat{B^\prime},R),\]
induced by $\Sigma$ is such that
\[\Phi_{\Sigma}(\beta_{\mathcal{F}}(B)) = \beta_{\mathcal{F}}(B^\prime)\qquad \Phi_{\Sigma}(\overline{\beta}_{\mathcal{F}}(B)) =\overline{\beta}_{\mathcal{F}}(B^\prime).  \]
Moreover, if $\Sigma$ consists of a single negative stabilization then
\[(x_1 - x_2) \Phi_{\Sigma}(\beta_{\mathcal{F}}(B)) = \pm\beta_{\mathcal{F}}(B) + d_{\mathcal{F}} x\quad (x_1 - x_2) \Phi_{\Sigma}(\overline{\beta}_{\mathcal{F}}(B^\prime)) =\pm \overline{\beta}_{\mathcal{F}}(B^\prime) + d_{\mathcal{F}} y,  \]
for some $x,\: y \in C_{\mathcal{F}}^\bullet(B^\prime)$, where $x_1$ and $x_2$ are defined as in Subsection \ref{Sec:exFA}.
\end{theorem}

\begin{remark}
In this paper, the maps associated to a generating set of Reidemeister moves are those commonly used through the literature (see for example \cite{ BarNatan05cob,Khovanov00, Turner04}). The maps associated to the other Reidemeister moves are obtained by composition. 
\end{remark}

The cycles  $ \beta_{\mathcal{F}}(B)$ and $ \overline{\beta}_{\mathcal{F}}(B)$ are called $\beta_{\mathcal{F}}$-invariants. The cycle $\psi(B)$ underlying the Plamenevskaya invariant belongs to this family since $\beta_{Kh}(B) = \overline{\beta}_{Kh}(B) = \psi (B)$. With an abuse of notation we refer to $\psi(B)$ as the Plamenevskaya invariant, and denote by $[\psi]$ the original Plamenevskaya invariant when confusion can arise. Similarly, the LNS invariants $\psi^{+}(B)$ and $\psi^{-}(B)$ are $\beta_{TLee}(B)$ and $\overline{\beta}_{TLee}(B)$, respectively.
Motivated by the common properties between the Plamenevskaya and the LNS invariants, we give the following definition.

\begin{definition}
Given a Frobenius algebra $\mathcal{F}$, a \emph{Plamenevskaya-type invariant} $x_{\mathcal{F}}$ is a function associating to each braid $B$ a chain $x_{\mathcal{F}}(B)\in C^{\bullet}_{\mathcal{F}}(\widehat{B},R)$, such that:
\begin{itemize}
\item[$\triangleright$] $x_{\mathcal{F}}(B)$ is non trivial;
\item[$\triangleright$] $x_{\mathcal{F}}(B)$ is an enhanced state, and its underlying resolution is the oriented resolution (see Subsection \ref{Subs:KhTypeHomologies} for the definitions);
\item[$\triangleright$] $x_{\mathcal{F}}(B)$ is invariant under transverse Markov moves in the sense of Theorem \ref{theorem:main1}.
\end{itemize}
\end{definition}
It turns out that the $\beta_{\mathcal{F}}$-invariants are essentially all the Plamenevskaya-type invariants. The second main result in this paper is the following
\begin{theorem}\label{theorem:main2}
Let $\mathcal{F}$ be a Frobenius algebra (as in Subsection \ref{Sec:exFA}) over $R$, and let $R$ be an UFD. Given a Plamenevskaya-type invariant $x_{\mathcal{F}}$, for each braid $B$ there exists $r = r(B) \in R$ such that: either $x_{\mathcal{F}}(B) = r \beta_{\mathcal{F}}(B)$ or $x_{\mathcal{F}}(B) = r \overline{\beta}_{\mathcal{F}}(B)$.
 In particular, there is a bijection between Plamenevskaya-type invariants and $R$-valued invariants for transverse braids.
\end{theorem}

The usage of the $\beta_{\mathcal{F}}$-invariants to distinguish transverse links is quite far from being practical. To distinguish two transverse links, one must verify that all the chain homotopy equivalences induced by sequences of either Markov or Reidemeister moves do not preserve the $\beta_{\mathcal{F}}$-invariants (see Subsection \ref{Subs:betaFinvariants}). 
In order to obtain a computable invariant, in Section \ref{sec:betaBN}, we set $\mathcal{F} = BN$. The corresponding $\beta$($=\beta_{BN}$)-invariants are equivalent to the LNS invariants. More precisely, we prove the following proposition.

\begin{proposition}\label{prop:equivalence}
Let $\Sigma$ be a sequence of Markov (resp. Reidemeister) moves from a braid $B$ to a braid $B^\prime$ (resp. from $\widehat{B}$ to $\widehat{B^\prime}$), and suppose that $sl(B) = sl(B^\prime)$. Denoted by
\[\Phi_{\Sigma}:C_{BN}^{\bullet,\bullet}(\widehat{B})\longrightarrow C_{BN}^{\bullet,\bullet}(\widehat{B^\prime})\quad \text{and}\quad \phi_{\Sigma}:C_{TLee}^{\bullet}(\widehat{B})\longrightarrow C_{TLee}^{\bullet}(\widehat{B^\prime})\]
the chain maps induced by $\Sigma$. Then,
\[ \Phi_{\Sigma}(\beta(B)) = \beta(B^\prime)\  \Leftrightarrow \  \phi_{\Sigma}(\psi^{+}(B)) = \psi^{+}(B^\prime)\]
and
\[ \Phi_{\Sigma}(\overline{\beta}(B)) = \overline{\beta}(B^\prime)\ \Leftrightarrow \   \phi_{\Sigma}(\psi^{-}(B)) = \psi^{-}(B^\prime)\]
\end{proposition}

To extract new information from the $\beta$-invariants we make use of the $\mathbb{F}[U]$-module structure of Bar-Natan homology. In particular, we define two (non-negative) integer invariants, called $c$-invariants. The $c$-invariants determine the vanishing of the Plamenevskaya invariant $[\psi(B)]$ (Proposition \ref{prop:psiebeta}). However, \emph{a priori} the $c$-invariants contain more information than the vanishing of $\psi$.

\begin{remark}
There is a class of knots (non-right-veering knots, see \cite{Plamenevskaya15}) for which both $[\psi]$ and $\widehat{\theta}$ vanish on all braid representatives. In particular, both $[\psi]$ and $\widehat{\theta}$ do not provide any information on transverse representatives of these knots. It will be interesting to understand which kind of information the $c$-invariants can provide in this case. As of now it is complicated to find examples of non-equivalent transverse braids representing non-right-veering knots, whose Khovanov homology (and $c$-invariants) can be computed.
\end{remark}

Building on the result of Lipshitz, Ng and Sarkar concerning the flype invariance (in the weak sense of Proposition \ref{prop:LipshitzNgSarkar}) of $\psi^{\pm}$, we prove our third main result.

\begin{proposition}\label{proposition:main3}
The $c$-invariants are invariant under negative flypes.
\end{proposition}

\begin{remark}
We prove a slightly stronger statement: given two braids related by a negative flype, there exists an isomorphism of bi-graded modules between their Bar-Natan homologies which preserves the homology classes of the $\beta$-invariants (see Proposition \ref{prop:flype invariance} for a precise statement).
\end{remark}

Moreover, we provide some sufficient conditions for the $c$-invariants to be non effective. We say that a link $\lambda$ is \emph{$c$-simple} if all transverse links representing $\lambda$ with the same self-linking number have the same $c$-invariants.

\begin{proposition}\label{corollary:csimpleknotsKh}
Let $\kappa$ be an oriented knot. If $\kappa$ satisfies one of the following conditions
\begin{enumerate}
\item $H_{Kh}^{0,j}(\kappa,\mathbb{F}) \equiv 0$  for each $j < s(\kappa) - 1$;
\item $H_{Kh}^{-1,j}(\kappa,\mathbb{F}) \equiv 0$  for each $j < s(\kappa) - 3$;
\item the torsion sub-module of $H^{0,\bullet}_{BN}(\kappa,\mathbb{F}[U])$ is isomorphic to the $\mathbb{F}[U]$-module
\[M = \bigoplus_{i=1}^{m} \frac{\mathbb{F}[U]}{(U^{2k_{i}})},\]
for some $m$, $k_{1}$, ..., $k_{m}\in\mathbb{N}\setminus \{ 0 \}$, and $H_{Kh}^{-1,j}(\kappa,\mathbb{F}) \equiv 0$ for each $j < s(\kappa) - 5$;
\end{enumerate}
where $s(\kappa) = s(\kappa, \mathbb{F})$ is the Rasmussen invariant (\cite{Rasmussen10}), then $\kappa$ is $c$-simple. In particular, all $Kh$-thin and $Kh$-pseudo-thin (i.e. $H_{Kh}^{0,\bullet}(\kappa, \mathbb{F})$ is concentrated in two quantum degrees) knots are $c$-simple. Moreover, in the cases (1)-(3) we have
\[ s(\kappa, \mathbb{F}) -1 = sl(B) +2c_{\mathbb{F}}(B),\]
for each braid representative $B$ of $\kappa$.
\end{proposition}

Since $[\psi(B)]=0$ if and only if $c_{\mathbb{F}}(B)>0$ (cf. Proposition \ref{prop:psiebeta}), we obtain the following result on the vanishing of $[\psi]$. The corresponding result in the special case of $Kh$-thin knots is due to Plamenevskaya (\cite[Theorem 1.2]{Plamenevskaya15}).

\begin{corollary}\label{cor:psine0}
Let $\kappa$ be a knot satisfying one among the conditions (1), (2) or (3) in Proposition \ref{corollary:csimpleknotsKh}. For each braid $B$ representing $\kappa$, we have $[\psi(B)] \ne 0$  if, and only if, $s(\kappa,\mathbb{F}) -1 = sl(B)$.
\end{corollary}

Finally, from the analysis of the Khovanov homology of all small knots (i.e. prime knots with less that $12$ crossings, up to mirror), we shall prove that they satisfy the hypotheses of Proposition \ref{corollary:csimpleknotsKh} (see Corollary \ref{corollary:allknotswithlessthan12arecsimple}). As a consequence, we obtain their $c$-simplicity and the following result.

\begin{corollary}\label{corollary:vanishingofpsiineff}
Let $\mathbb{F}$ be a field with $char(\mathbb{F}) \ne 2$, and let $\kappa$ be (up to mirror) a prime knot with less than $12$ crossings. Given a braid $B$ representing $\kappa$, the class $[\psi(B)]$ is trivial (over $\mathbb{F}$) if, and only if, $s(\kappa,\mathbb{F}) -1 = sl(B)$. In particular, the vanishing of $[\psi]$ (over $\mathbb{F}$) is a non-effective invariant for all such $\kappa$'s.
\end{corollary}

\subsection*{Acknowledgments}
The author wishes to thank Prof. Paolo Lisca the helpful conversations and his continuous support, Alberto Cavallo for his careful reading of a draft of this paper, and Andr\'e Belotto for his advice. Many thanks go to Daniele Angella for his patience and support. The author is grateful to the anonymous referees for their helpful suggestions. The content of the present paper is a part of the author's PhD thesis, during the writing of which he was supported by a PhD scholarship ``Firenze-Perugia-Indam''.

\section{Frobenius Algebras}
This section contains some basic material concerning Frobenius algebras, and is divided into two subsections. The first subsection concerns the definition of graded and filtered Frobenius algebras. The second subsection is devoted to the description of some examples which play a central role in this paper.
\subsection{Definitions}

A \emph{Frobenius algebra} $\mathcal{F}$, over the ring $R_{\mathcal{F}}$, is a commutative unitary $R_{\mathcal{F}}$-algebra $A_{\mathcal{F}}$, together with two maps
\[ \Delta_{\mathcal{F}}: A_{\mathcal{F}} \to A_{\mathcal{F}}\otimes_{R_{\mathcal{F}}} A_{\mathcal{F}},\quad \epsilon_{\mathcal{F}}: A_{\mathcal{F}} \to R_{\mathcal{F}}, \]
satisfying the following requirements
\begin{enumerate}[(a)]\label{def:FA}
\item $A_\mathcal{F}$ is a finitely generated and projective $R_\mathcal{F}$-module;
\item \label{defitem:algandco-alg}  $\Delta_{\mathcal{F}}$ is an $A_{\mathcal{F}}$-bi-module isomorphism (i.e.\: commutes with the left and right action of $A_{\mathcal{F}}$ over $A_{\mathcal{F}}\otimes A_{\mathcal{F}}$);
\item $\epsilon_{\mathcal{F}}$ is $R_{\mathcal{F}}$-linear;
\item $\Delta_{\mathcal{F}}$ is co-associative, that is $(id\otimes \Delta_{\mathcal{F}} ) \circ \Delta_{\mathcal{F}} = (\Delta_{\mathcal{F}} \otimes id ) \circ \Delta_{\mathcal{F}}$;
\item $\Delta_{\mathcal{F}}$ is co-commutative, that is $\tau \circ \Delta_{\mathcal{F}} = \Delta_{\mathcal{F}}$, where $\tau (a\otimes b) = b \otimes a$;
\item $(id_{A_{\mathcal{F}}} \otimes \epsilon_{\mathcal{F}})\circ \Delta_{\mathcal{F}} = id_{A_{\mathcal{F}}} = ( \epsilon_{\mathcal{F}}\otimes id_{A_{\mathcal{F}}})\circ \Delta_{\mathcal{F}} $.
\end{enumerate}
The map $\Delta_{\mathcal{F}}$ is called \emph{co-multiplication}, while $\epsilon_{\mathcal{F}}$ is the \emph{co-unit} relative to $\Delta_{\mathcal{F}}$.

Whenever we deal with more than one Frobenius algebra at once, we will usually keep the subscript indicating to which Frobenius algebra the maps $\Delta$, $\epsilon$, the algebra $A$, and the ring $R$ belong to. Sometimes, it will be necessary to specify the multiplicative structure on $A_{\mathcal{F}}$, so we will denote by $m_\mathcal{F}$ the ($R_\mathcal{F}$-linear) multiplication map from $A_\mathcal{F} \otimes_{R_\mathcal{F}} A_{\mathcal{F}}$ to $A_{\mathcal{F}}$. Finally, we will denote by $\iota_{\mathcal{F}}$ the \emph{unit map}, that is the $R_{\mathcal{F}}$-linear map from $R_{\mathcal{F}}$ to $A_{\mathcal{F}}$ sending $1_{R_{\mathcal{F}}}$ to $1_{A_{\mathcal{F}}}$.

We require the graded and filtered versions of Frobenius algebras.
\begin{definition}\label{def:gradedFrobeniusalg}
A \emph{graded Frobenius algebra} is a Frobenius algebra $\mathcal{F}$, satisfying the following properties
\begin{enumerate}[(a)]
\item $R_{\mathcal{F}} = \bigoplus_{k} R_{k}$ is a graded ring;
\item $A_{\mathcal{F}}= \bigoplus_{i} A_{i}$ is a graded $R$-module;
\item $\epsilon_{\mathcal{F}},\ \iota_{\mathcal{F}}$ are graded maps;
\item $\Delta_{\mathcal{F}},\ m_{\mathcal{F}}$ are graded maps (where $A_{\mathcal{F}}\otimes A_{\mathcal{F}}$ is given the usual tensor grading);
\end{enumerate}
\end{definition}

\begin{definition}\label{def:filteredFrobeniusalg}
A \emph{filtered Frobenius algebra} is a Frobenius algebra $\mathcal{F}$ over a (possibly trivially) graded ring $R$ together with a filtration $\mathscr{F}_{\circ}$ of $A = A_{\mathcal{F}}$ as an $R$-module, for which there exists an integer $d$ such that: 
\[ \mathscr{F}_{i}\mathscr{F}_{j} \subseteq \mathscr{F}_{j+i+d}\]
for each $i$ and each $j$, and
\[ \Delta_{\mathcal{F}}(\mathscr{F}_{n}) \subseteq \sum_{k} \mathscr{F}_{k} \otimes \mathscr{F}_{n-d-k} \subseteq A \otimes A, \]
for each $n$.
\end{definition}

\begin{definition}\label{def:morphismFA}
Let $\mathcal{F} = (R,A,m,\iota,\Delta,\epsilon)$, and $\mathcal{G} = (R^\prime,B,n,\jmath,\Gamma,\eta)$ be two (graded) Frobenius Algebras. A \emph{Frobenius algebra morphism} is a pair of ring homomorphisms
\[ f: R \to R^\prime,\quad F:A \to B,\]
such that
\[ F \circ \iota   =  \jmath \circ f,\quad \eta \circ F = f \circ \epsilon. \]
and
\[ F\otimes F \circ \Delta = \Gamma \circ F.\]
Given two Frobenius algebra morphisms, say $(f,F)$ and $(g,G)$, their \emph{composition} is defined as $(f \circ g, F \circ G )$. An \emph{isomorphism of Frobenius algebras} is a morphism $(f,F)$ such that both $f$ and $F$ are ring isomorphisms. 
\end{definition}

\subsection{Examples}\label{Sec:exFA}

A large family of Frobenius algebras, which plays a central role in this paper, is defined as follows
\[ A_\mathcal{F} = \frac{R_\mathcal{F}[X]}{(X^{2} + P X + Q)}\quad R_\mathcal{F} = \frac{\mathbb{F}[U,V]}{(p(U,V),q(U,V))}\]
with $p$ and $q$ such that $(p,q)$ is a (possibly trivial) prime ideal in $\mathbb{F}[U,V]$.

Up to twist equivalence\footnote{Given a Frobenius algebra $\mathcal{F}$  we can \emph{twist} its co-multiplication and its co-unit by pre-composing them with the multiplication by an invertible element in $A_{\mathcal{F}}$. Two Frobenius algebras are twist equivalent if the can be obtained one from the other via twist. It has been proven by Khovanov \cite[Proposition 3 and Corollary 1]{Khovanov06} that two twist equivalent Frobenius algebras of rank $2$ produce isomorphic Khovanov-type homologies.} (\cite[Theorem 1.6]{Kadison99}, see also \cite{Khovanov06}) we may assume
\[\epsilon_\mathcal{F}(X) = 1_{R_\mathcal{F}},\quad \epsilon_\mathcal{F}(1_{A_\mathcal{F}}) = 0.\]
For some technical reasons which will be made clear in the next section, it is necessary to have zero divisors in $A_\mathcal{F}$. This implies that $X^{2} + P X + Q$ factors over $R_\mathcal{F}$. Thus, we may write
\[ (X^{2} + P X + Q) = (X- x_1) (X - x_2), \]
where $x_1 ,\: x_2 \in R_{\mathcal{F}}$.
\begin{remark}
Notice that the case $x_1 = x_2$ is not excluded.
\end{remark}
Set
\[ x_\circ = (X- x_1) \quad \text{and} \quad x_\bullet = (X- x_2).\]
In the rest of the paper we will be using the properties of $x_\circ$ and $x_\bullet$. In particular, we need to understand the behaviour of $x_\circ$ and $x_{\bullet}$ with respect to the operations $\Delta_{\mathcal{F}}$, $m_{\mathcal{F}}$ and $\epsilon_{\mathcal{F}}$. First, notice that
\begin{equation}
m_\mathcal{F}(x_\circ,x_\circ) = -(x_1 - x_2)x_\circ,\quad  m_\mathcal{F}(x_\bullet,x_\bullet) = (x_1 - x_2)x_\bullet,
\label{eq:productandcirclebullet1}
\end{equation}
\begin{equation}
 m_\mathcal{F}(x_\circ,x_\bullet) = 0,
\label{eq:productandcirclebullet2}
\end{equation}
and 
\begin{equation}
\quad \epsilon_\mathcal{F}(x_\circ) = \epsilon_\mathcal{F}(x_\bullet)=1_{R_\mathcal{F}}.
\label{eq:counitandcirclebullet}
\end{equation}
Furthermore, by setting
\[ e_{x_*} = \begin{cases} \phantom{-}1 & * = \circ \\ -1 & * = \bullet \end{cases},\]
and
\[ \overline{x_{\circ}} = x_\bullet,\qquad \overline{x_\bullet} = x_\circ,\]
we obtain
\begin{equation}
\overline{x} = x - e_x (x_2 - x_1)1_{A_\mathcal{F}},\quad x\in \{ x_\circ,\: x_\bullet\}.
\label{eq:differencebetweentheconjugatesgeneral}
\end{equation}
The involution $\overline{x}$ in the set $\{ x_\circ , x_\bullet\}$ will be called \emph{conjugation}.

Given a Frobenius algebra $(A,m,\iota,\Delta,\epsilon)$ with $x\in A$, write
\begin{equation}
 \Delta(x) = \sum_{i} x^\prime_i \otimes x^{\prime\prime}_i
\label{eq:relationmultcomultdeltaF1}
\end{equation}
where the elementary tensors $x^\prime_i \otimes  x^{\prime\prime}_i$ are totally determined by the equations:
\begin{equation}
 m(x,y) = \sum_{i} (x^{\prime\prime}_i,y) x^\prime_i,\quad \forall y \in A, 
\label{eq:relationmultcomultdeltaF2}
\end{equation}
and $(\cdot, \cdot)$ indicates the (non-degenerate) bi-linear pairing $\epsilon(m(\cdot,\cdot))$ (\cite[Propositions 4.3 and 4.8]{Kadison99}, see also \cite[Chapter 2]{Turner06}).
It follows immediately from Equations \eqref{eq:productandcirclebullet1}, \eqref{eq:productandcirclebullet2}, \eqref{eq:relationmultcomultdeltaF1} and \eqref{eq:relationmultcomultdeltaF2} that
\begin{equation}
 \Delta_\mathcal{F} (x_\bullet ) = x_\bullet \otimes x_\bullet,\quad \Delta_\mathcal{F}(x_\circ) = x_\circ \otimes x_\circ.
  \label{eq:comultandcirclebullet}
\end{equation}

Finally, we need to check the \emph{de-cupped torus map}, that is the $ R_{\mathcal{F}}$-linear map
\[t_\mathcal{F}:\: R_{\mathcal{F}} \longrightarrow A_{\mathcal{F}}:\: 1_{R_{\mathcal{F}}} \longmapsto  m_\mathcal{F}(\Delta_\mathcal{F}(1_{A_\mathcal{F}})) \]
Simple computations show that
\[ \Delta_\mathcal{F}(1_{A_{\mathcal{F}}}) = x_{\circ} \otimes 1_{A_\mathcal{F}} + 1_{A_\mathcal{F}} \otimes x_{\bullet} =  x_{\bullet} \otimes 1_{A_\mathcal{F}} + 1_{A_\mathcal{F}} \otimes x_{\circ},\]
from which it follows
\begin{equation}
t_{\mathcal{F}}(1_{R_{\mathcal{F}}}) = x_\bullet + x_\circ.  
\label{eq:torusmap}
\end{equation}
To conclude this section, we shall list some special elements of the family $\mathcal{F}$. Let $R$ be a ring. Define $A_{BIG}$ to be the (graded) $R[U,V]$-algebra
\[ A_{BIG} = \frac{R[U,V][X]}{(X^{2} - UX + V)},\]
where $x_{-} := X$, and $x_{+} := 1$, have degrees, respectively, $-1$ and $+1$. 
In order to define the structure of Frobenius algebra, define a co-multiplication $\Delta = \Delta_{BIG}$, as follows
\[ \Delta(x_{+}) = x_{+}\otimes_{R} x_{-} + x_{-}\otimes_{R} x_{+} - U x_{+}\otimes x_{+}, \]
\[ \Delta(x_{-}) = x_{-} \otimes_{R} x_{-} - V x_{+} \otimes x_{+}. \]
Finally, the co-unit map is defined by
\[ \epsilon :A_{BIG} \to R[U,V]: a(U,V) x_{+} + b (U,V) x_{-} \mapsto b(U,V). \]

Many familiar examples are obtained by specifying $U$, $V$ or both, in elements $u$ or $v$ of $R$ (that is, applying the functor $\cdot \otimes_{R[U,V]} R[U,V]/(U-u, V-v)$). In particular
\begin{enumerate}
\item Khovanov theory $Kh$, is obtained by setting $ U = 0$, $V=0$;
\item the original Lee theory, denoted by $OLee$, is obtained by setting $U=0$ and $V= 1$;
\item the twisted Lee theory (also known as filtered Bar-Natan theory), denoted by $TLee$, is obtained by setting $U = 1$ and $V=0$;
\item the Bar-Natan theory, denoted by $BN$, is obtained by setting $V= 0$;
\item the $V$-theory, denoted by $VT$, is obtained by setting $U = 0$.
\end{enumerate}
If one defines
\[ deg(U)= -2,\quad \text{and}\quad deg(V) = -4,\]
$BIG$ becomes a graded Frobenius algebra and hence $BN$, $VT$, and $Kh$, inherit this structure, while $TLee$ and $OLee$ become filtered Frobenius algebras.

\section{Khovanov-type homologies} \label{Sec:Khovanovhom}

The aim of this section is to recall the definition of Khovanov-type homology. We shall review first how to associate to each Frobenius algebra $\mathcal{F}$ (over the ring $R$) and to each link diagram $D$, a chain complex $C^\bullet_{\mathcal{F}}(D,R)$. The homotopy type of this chain complex depends only on the link and not on the diagram $D$ representing it. The second part of this section shall be concerned on reviewing how to endow the complex $C_{\mathcal{F}}^\bullet(D,R)$ with either a second grading or a filtration, depending on whether $\mathcal{F}$ is a graded or filtered Frobenius algebra. Finally, the end of the section shall be dedicated to fix some notation used in the rest of the paper.

\subsection{Definition of Khovanov-type homologies}\label{Subs:KhTypeHomologies}

Let $D$ be an oriented link diagram. A \emph{local resolution} of a crossing \ocrosstext is its replacement with either a $0$-resolution \isplittext or with a $1$-resolution \osplittext .

\begin{definition}
A \emph{resolution of} $D$ is the set of circles, embedded in $\mathbb{R}^{2}$, obtained from $D$ by performing a local resolution at each crossing. The total number of $1$-resolutions performed in order to obtain a resolution $\underline{s}$ is denoted by $\vert \underline{s}\vert$.
\end{definition}

Let $\mathcal{R}_D$ be the set of all the possible resolutions of $D$. Define an elementary relation on $\mathcal{R}_{D}$ as follows
\[ \underline{r} \prec \underline{s} \iff \vert \underline{r} \vert < \vert \underline{s} \vert\ \text{and } \underline{r},\ \underline{s}\ \text{differ by a single local resolution.} \]
A \emph{square} $[s_{0},s_{1},s_{2},s_{3}]$ is a collection of four (distinct) resolutions such that: $\underline{s}_{0} \prec \underline{s}_{1}$, $  \underline{s}_{0} \prec \underline{s}_{2} $, $\underline{s}_{1} \prec \underline{s}_{3}$, and $\underline{s}_{2} \prec \underline{s}_{3}$.
\begin{definition}
A \emph{sign function} is a map 
\[ \mathsf{sgn}: \mathcal{R}_{D} \times \mathcal{R}_{D} \to \{ 0,1,-1\},\]
satisfying the following two properties:
\begin{enumerate}
\item $\mathsf{sgn}(\underline{r},\underline{s}) = 0$ if, and only if, $\underline{r} \nprec \underline{s} $;
\item for each square $[s_{0},s_{1},s_{2},s_{3}]$, we have
\[ \mathsf{sgn}(\underline{s}_{0},\underline{s}_{1}) \mathsf{sgn}(\underline{s}_{1},\underline{s}_{3}) = - \mathsf{sgn}(\underline{s}_{0},\underline{s}_{2}) \mathsf{sgn}(\underline{s}_{2},\underline{s}_{3}).\]
\end{enumerate}
\end{definition}
Given a Frobenius algebra $\mathcal{F} = (R,A,m,\iota, \Delta,\epsilon)$ define
\[ C_{\mathcal{F}}^{i}(D,R) = \bigoplus_{\vert \underline{r} \vert - n_{-}(D) = i} A_{\underline{r}},\quad A_{\underline{r}} = \bigotimes_{\gamma\in \underline{r}} A_{\gamma}, \]
where $A_{\gamma}$ is just a copy of $A$ indexed by a circle $\gamma \in \underline{r}$, $n_{-}(D)$ (resp. $n_{+}(D)$) denotes the number of negative (resp. positive) crossings in $D$, and $\underline{r}$ runs through $\mathcal{R}_{D}$. These are the $R$-modules which play the role of (co)chain groups. In order to define a (co)chain complex, all that is left to do is to define a differential. This is done in two steps. Start by defining
\[ d_{\underline{r}}^{\underline{s}}: A_{\underline{r}} \to A_{\underline{s}},\quad \underline{r}\prec \underline{s}.\]
 Given $\underline{s}$ such that $\underline{r}\prec \underline{s}$, then $\underline{r}$ and $\underline{s}$ differ by a single local resolution. Hence there is an identification of all the circles in the two resolutions, except the ones involved in the change of local resolution. There are only two cases to consider: (a) two circles of $\underline{r}$, say $\gamma_{1},\ \gamma_{2}$ are merged in a single circle $\gamma^\prime_{1}$ in $\underline{s}$, or (b) a circle $\gamma_{1}$ belonging to $\underline{r}$ is split in into two circles, say $\gamma_{1}^\prime$ and $\gamma_{2}^\prime$, in $\underline{s}$. Consider the elementary tensor $x = \bigotimes_{\gamma\in \underline{r}} \alpha_{\gamma} \in A_{\underline{r}}$, then $d_{\underline{r}}^{\underline{s}}$ is defined as follows
\[ d_{\underline{r}}^{\underline{s}} (x) = \begin{cases}\bigotimes_{\gamma\in \underline{r}\cap \underline{s}} \alpha_{\gamma} \otimes m(\alpha_{\gamma_{1}},\alpha_{\gamma_{2}}) & \text{in case (a)}\\
\bigotimes_{\gamma\in \underline{r}\cap \underline{s}} \alpha_{\gamma} \otimes \Delta(\alpha_{\gamma_{1}}) & \text{in case (b)} \\\end{cases}\]
and extended by $R$-linearity.
Finally, fix a sign function $\mathsf{sgn}$ and define
\[ d_{\mathcal{F}}^{i}: C^{i}_{\mathcal{F}}(D,R) \to C^{i+1}_{\mathcal{F}}(D,R): x\in A_{\underline{r}} \mapsto \sum_{\underline{r}\prec\underline{s}} \mathsf{sgn}(\underline{r},\underline{s})d_{\underline{r}}^{\underline{s}}(x).\]

\begin{remark}
Notice that $d_{\underline{r}}^{\underline{s}}$ is well-defined because of the commutativity of $m$, and of the co-commutativity of $\Delta$. On the other hand, $d_{\mathcal{F}}$ depends on the choice sign function $\mathsf{sgn}$. In particular, the existence of $d_{\mathcal{F}}$ depends on the existence of such a function.
\end{remark}

\begin{proposition}[Khovanov, \cite{Khovanov00}]
There exists a sign function $\mathsf{sgn}$ such that the complex $(C_{\mathcal{F}}^{\bullet}(D,R),d^{\bullet}_{\mathcal{F}})$ is a (co)chain complex. Furthermore, $(C_{\mathcal{F}}^{\bullet}(D,R),d^{\bullet}_{\mathcal{F}})$ does not depend, up to isomophism, on the choice of the sign function $\mathsf{sgn}$, or on the order of the circles in each resolution. Finally, given a sequence of Reidemeister moves between two diagrams $D$ and $D^\prime$, there exists a chain homotopy equivalence
\[\Phi: C_{\mathcal{F}}^{\bullet}(D,R) \longrightarrow C_{\mathcal{F}}^{\bullet}(D^\prime,R)\]
induced by this sequence.
\end{proposition}

\begin{remark}
The map associated to the sequence of Reidemeister moves is not uniquely defined (see \cite{Jacobsson04}). In Section \ref{sec:invariancebeta} we specify the maps associated to each Reidemeister move used in this paper. These maps are the maps commonly used through the literature (for example, see \cite{BarNatan05cob}).
\end{remark}

\begin{remark}\label{rem:disjointunionandtensor}
It is immediate from the definition of Khovanov-type homology that
\[ C^{\bullet}_{\mathcal{F}}(D\sqcup D^\prime, R) \simeq  C^{\bullet}_{\mathcal{F}}(D, R)\otimes_{R}  C^{\bullet}_{\mathcal{F}}(D^\prime, R), \]
as complexes of $R$-modules. Moreover, if $\mathcal{F}$ is a graded (resp. filtered) Frobenius algebra the above isomorphism respects the quantum grading (resp. the filtration) defined in the next subsection.
\end{remark}
\subsection{Gradings and some notation}\label{Subs:GradingC_F}

There is a unique condition on $\mathcal{F}$ for the homology of the complex $(C_{\mathcal{F}}^{\bullet}(D,R),d^{\bullet}_{\mathcal{F}})$ to yeld a link invariant: the rank of $A_{\mathcal{F}}$ being $2$ (cf. \cite[Proposition 3, Theorems 5 and 6]{Khovanov06}). Moreover, this condition is also sufficient to get functoriality (up to sign and up to boundary fixing isotopy \cite{BarNatan05cob, Jacobsson04}).

As we are concerned only with link invariant theories, from now on all Frobenius algebras are supposed to be free of rank $2$. Once a basis of $A$ is fixed, say $x_{+},\ x_{-}$, the elements of $C^{i}_{\mathcal{F}}(D,R)$ of the form $\bigotimes_{\gamma\in\underline{r}} \alpha_{\gamma}$, with $\alpha_{\gamma} \in \{x_{+},\ x_{-}\}$ and $\underline{r}\in \mathcal{R}_D$, will be called \emph{states}; while those of the form $\bigotimes_{\gamma\in\underline{r}} \alpha_{\gamma}$, with $\alpha_\gamma \in A$, will be called \emph{enhanced states}. Notice that the states are an $R$-basis of $C^{\bullet}_{\mathcal{F}}(D,R)$, while the enhanced states are a system of generators. 

\begin{remark}\label{rem:qdeg}
If $\mathcal{F}$ is a graded (resp. filtered) Frobenius algebra, then the basis $\{ x_{+},x_{-}\}$ will be taken to be composed of homogeneous elements (resp. to be a filtered basis).
Under these conditions, it is possible to define another grading (resp. filtration) over the complex $(C^{\bullet}_{\mathcal{F}}(D,R),d^{\bullet}_{\mathcal{F}})$, as follows
\[ qdeg(\bigotimes_{\gamma\in \underline{r}} \alpha_{\gamma}) = \sum_{\gamma\in\underline{r}} deg_{A}(\alpha_{\gamma}) - 2 n_{-}(D) + n_{+}(D)+ \vert \underline{r} \vert,\]
for each state $\bigotimes_{\gamma\in\underline{r}} \alpha_{\gamma}$. (Then the filtration is given by considering all the elements which can be written as combination of states of degree greater or lower than  a fixed $qdeg$, depending on whether the multiplication is non-decreasing or non-increasing with respect to the $qdeg$.) Moreover, by definition of graded (resp. filtered) Frobenius algebra, the differential $d_{\mathcal{F}}^\bullet$ is homogeneous (resp. filtered) with respect to the $qdeg$ degree (resp. induced filtration), and the resulting homology theory is hence doubly-graded (resp. filtered). Let $\mathcal{F}$ be a filtered Frobenius algebra, the filtration induced on the complex $C_\mathcal{F}^\bullet(D,R)$ is denoted by $\mathscr{F}_{\circ}C_\mathcal{F}^\bullet(D,R)$.
\end{remark}

\begin{theorem}
Let $D$ be an oriented link diagram. If $\mathcal{F}$ and $\mathcal{G}$ are isomorphic (graded, resp. filtered) Frobenius algebras, then $(C^\bullet_{\mathcal{F}}(D,R_\mathcal{F}),d^\bullet_{\mathcal{F}})$ and $(C^\bullet_{\mathcal{G}}(D,R_\mathcal{G}),d^\bullet_{\mathcal{G}})$ are isomorphic as (doubly-graded, resp. filtered) complexes of both $R_\mathcal{F}$ and $R_\mathcal{G}$ modules.
\end{theorem}
\begin{proof}
Let $\mathcal{F} = (R_\mathcal{F},A,m,\iota,\Delta,\epsilon)$ and $\mathcal{G} = (R_\mathcal{G},B,n,\jmath,\gamma,\eta)$ be two isomorphic (graded, resp. filtered) Frobenius algebras, and let $(f,F)$ the (graded, resp. filtered) isomorphism between them. Then, for each resolution $\underline{r}$ we have the isomorphism\footnote{This is injective because $A$, $B$ are both flat $R_\mathcal{F}$-modules, and is obviously surjective.} of (graded, resp. filtered) $R_\mathcal{F}$-modules 
\[ \bigotimes_{\gamma\in \underline{r}} F:A_{\underline{r}} \to B_{\underline{r}},\]
where $B$ is seen as an $R_\mathcal{F}$-module with the induced structure. This naturally induces an isomorphism of (bi-graded, resp. filtered) chain modules that commutes (by definition of morphism between Frobenius algebras) with differentials. The same reasoning works if $R_\mathcal{F}$ is replaced by $R_\mathcal{G}$.
\end{proof}

\begin{remark}
Until now we have required the diagrams to be oriented: this is essential for the invariance. As the reader may have noticed, the orientation comes up in the degree shift. The homological degree has been shifted by the number of negative crossings. Without this shift the homology is not invariant as graded module (much less as bi-graded or filtered module).
\end{remark}

Let $D$  be an oriented link diagram, and let $\mathbf{a}\subseteq \mathbb{R}^{2}$ be a segment joining two strands of $D$ (i.e. edges of the underlying graph), and meeting $D$ only at the endpoints. Therefore $\mathbf{a}$ does not intersect a crossing of $D$. Let $D^\prime$ be the unoriented link diagram formed by replacing a small neighbourhood of $\mathbf{a}$ as shown in Figure \ref{fig:saddle}. Given a resolution $\underline{r}$ of $D$, denote by $\gamma_1$ and $\gamma_2$ the circles in $\underline{r}$ (possibly $\gamma_1 = \gamma_2$) containing the endpoints of $\mathbf{a}$.
\begin{figure}[h]
\centering
\begin{tikzpicture}[scale=.25]{
\osplit{0}{0}
\draw[dashed] (2,0.75) -- (2,3.25);
\draw[dashed] (2,2) circle (2.8);
\node at (2.2,0) {$\gamma_{1}$};
\node at (1.8,4) {$\gamma_{2}$};
\node at (2.5,2) {$\mathbf{a}$};
\isplit{10}{0}
\draw[dashed] (12,2) circle (2.8);
\node at (2.2,-1.5) {$D$};
\node at (12.2,-1.5) {$D^\prime$};
}\end{tikzpicture}
\caption{The diagrams $D$ and $D^\prime$.}\label{fig:saddle}
\end{figure}
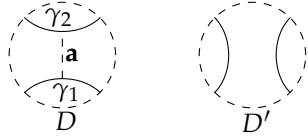 

Notice that there is not, in general, a canonical way to endow $D^\prime$ with an orientation compatible with the one of $D$. The existence of such an orientation depends on the relative orientations of the arcs containing the endpoints of \textbf{a} in $D$. We shall assume $D^\prime$ to be given such an orientation if it exists, otherwise we shall randomly orient $D^\prime$.

The \emph{saddle move along $\mathbf{a}$} is the map
\[\mathbf{S}: C_{\mathcal{F}}^{\bullet} (D,R) \longrightarrow C_{\mathcal{F}}^{\bullet - \omega(D,D^\prime)}(D^\prime,R),\]
where $\omega(D,D^\prime) = n_{-}(D) - n_{-}(D^\prime)$, and $\mathbf{S}$ is defined on enhance states as follows:
\[ \mathbf{S}(\bigotimes_{\gamma\in \underline{r}} \alpha_{\gamma}, \mathbf{a}) = \bigotimes_{\gamma\in \underline{r}\setminus \{ \gamma_1,\gamma_{2}\}} \alpha_{\gamma} \otimes \varsigma(\alpha_{\gamma_1} , \alpha_{\gamma_{2}}),\]
and 
\[\varsigma(\alpha_{\gamma_{1}},\alpha_{\gamma_{2}}) = \begin{cases} m_{\mathcal{F}}(\alpha_{\gamma_{1}},\alpha_{\gamma_{2}}) & \text{if}\: \gamma_{1}\neq\gamma_{2} \\ \Delta_{\mathcal{F}}(\alpha_{\gamma_{1}}) & \text{if}\: \gamma_{1}=\gamma_{2} \\\end{cases}\]

\begin{remark}
Notice that $\mathbf{S}$ is well defined because of the commutativity of $m_{\mathcal{F}}$ and of the co-commutativity of $\Delta_{\mathcal{F}}$. Moreover, given an enhanced state $x$, the chain $\mathbf{S}(x,\mathbf{a})$ is a sum of enhanced states and not necessarily a single enhanced state.
\end{remark}

\begin{remark}
If $\mathcal{F}$ is a graded (or a filtered) Frobenius algebra, then $\mathbf{S}$ is a graded (resp. filtered) map of (filtered) degree $-3\omega(D,D^\prime)-1$.
\end{remark}

\section{Transverse invariants in Khovanov-type homologies}\label{sec:invariancebeta}

The aim of this section is dual: to define $\beta_{\mathcal{F}}$-invariant, and prove some of their properties. This section is divided into six subsections. In the first subsection we introduce the $\beta_{\mathcal{F}}$-cycles. The second, third and fourth subsections are dedicated to the study of the behaviour of the $\beta_{\mathcal{F}}$-cycles with respect to the maps induced by the Reidemeister moves. In the fifth subsection the $\beta_{\mathcal{F}}$-cycles are used to define the $\beta_{\mathcal{F}}$-invariants. Moreover, we prove Theorem \ref{theorem:main1} and discuss in detail in which sense the $\beta_{\mathcal{F}}$-invariants are invariants for transverse links. Finally, the last subsection is dedicated to the proof of Theorem \ref{theorem:main2}, which concerns the uniqueness of the $\beta_{\mathcal{F}}$-invariants.

\subsection{The definition of the $\beta_{\mathcal{F}}$-cycles}
Let $D$ be an oriented link diagram, and denote by $\underline{r}$ the oriented resolution\footnote{The unique resolution which inherits an orientation from $D$.} of $D$. Mark a point $p_{\gamma}$ on each circle $\gamma$ in $\underline{r}$, and let $q_{\gamma}$ be the point in $\mathbb{S}^{2}$ obtained by pushing $p_{\gamma}$ slightly to the left with respect to the orientation on $\underline{r}$ induced by $D$. 
The \emph{nesting number}\label{Nestingnumber} $N (\gamma)$ is the number, counted modulo $2$, of intersection points between the circles in $\underline{r}$ and a generic segment between $q_{\gamma}$ and the point at the infinity in $\mathbb{S}^{2} = \mathbb{R}^{2} \cup \{ \infty \}$.

Define \emph{$\beta_\mathcal{F}$-cycles} as follows: $\beta_\mathcal{F}(D,R) \in C^{\bullet,\bullet}_{\mathcal{F}}(D,R)$ is the enhanced state with underlying resolution the oriented resolution, where each circle $\gamma$ has label (i.e. the factor corresponding to $A_{\gamma}$ in the tensor product)
\[ b_\gamma =  b_\gamma (\mathcal{F}) =
\begin{cases}
x_{\circ} & \text{if }N(\gamma)\equiv 0\ mod\: 2\\
x_\bullet & \text{if }N(\gamma)\equiv 1\ mod\: 2 \\
\end{cases}
\]
The chain $\overline{\beta}_{\mathcal{F}}(D,R)$ is defined exactly as $\beta_{\mathcal{F}}(D,R)$ but exchanging the roles of $x_{\circ}$ and $x_{\bullet}$. Sometimes the reference to $R$ will be omitted from both the notation for the $\beta_{\mathcal{F}}$-cycles and notation for the Khovanov-type chain complexes.

\begin{remark}
Notice that in general $\beta_{\mathcal{F}}(D,R)$ and $\overline{\beta}_{\mathcal{F}}(D,R)$ are not distinct.
\end{remark}

\begin{proposition}
Let $D$ be an oriented link diagram. The enhanced states $\beta_{\mathcal{F}}(D,R)$, $\overline{\beta}_{\mathcal{F}}(D,R)\in C_{\mathcal{F}}^{\bullet}(D,R)$ are cycles.
\end{proposition}
\begin{proof}
Since two circles in the oriented resolution share a crossing only if they have distinct nesting numbers (see \cite[Corollary 2.5]{Rasmussen10}), which also implies that each change of a local resolution in the oriented resolution merges two circles, the proposition follows directly from Equation \eqref{eq:productandcirclebullet2}.
\end{proof}

Henceforth all the results will be stated for $\beta_{\mathcal{F}}(D,R)$, with the understanding that the same results hold by replacing $\beta_{\mathcal{F}}(D,R)$ with $\overline{\beta}_{\mathcal{F}}(D,R)$.

\subsection{First Reidemeister move}
Let $D$ be an oriented link diagram. Denote by $D^\prime_+$ the oriented link diagram obtained from $D$ via a positive first Reidemeister move (i.e.\: the addition of a positive curl, see Figure \ref{reidmoves1}) on an arc \textbf{a}. Finally, denote by $c_+$ the crossing appearing only in $D_{+}^\prime$.

\begin{figure}[H]
\centering
\begin{tikzpicture}[scale = .75, thick]

\draw[->] (-1,0) .. controls +(.5,.5) and +(-.5,.5) ..  (1,0);

\node at (0,-.5) {$D$};
\node at (1,.35) {\textbf{a}};

\node at (4,-.5) {$D^\prime_+$};
\node at (-4,-.5) {$D^\prime_-$};

\node at (-2,1) {$R_1^-$}; 
\node at (-2,.5) {$\leftrightharpoons$}; 
\node at (2,1) {$R_1^+$}; 
\node at (2,.5) {$\rightleftharpoons$}; 

\draw[<-] (-3,0) .. controls +(-.5,.5) and +(.25,-.5) ..  (-4.5,1);
\pgfsetlinewidth{8*\pgflinewidth}
\draw[white] (-5,0) .. controls +(.5,.5) and +(-.25,-.5) ..  (-3.5,1);
\pgfsetlinewidth{.125*\pgflinewidth}
\draw[<-] (-4.5,1) .. controls +(-.25,.5) and +(.25,.5) ..  (-3.5,1);
\draw[->] (-5,0) .. controls +(.5,.5) and +(-.25,-.5) ..  (-3.5,1);

\begin{scope}[shift= {(8,0)}]
\draw[<-] (-4.5,1) .. controls +(-.25,.5) and +(.25,.5) ..  (-3.5,1);
\draw[->] (-5,0) .. controls +(.5,.5) and +(-.25,-.5) ..  (-3.5,1);
\pgfsetlinewidth{8*\pgflinewidth}
\draw[white] (-3,0) .. controls +(-.5,.5) and +(.25,-.5) ..  (-4.5,1);
\pgfsetlinewidth{.125*\pgflinewidth}
\draw[<-] (-3,0) .. controls +(-.5,.5) and +(.25,-.5) ..  (-4.5,1);
\end{scope}

\node at (4.75,.5) {$c_+$};
\node at (-4.75,.5) {$c_-$};

\end{tikzpicture}
\caption{The first Reidemeister move.}
\label{reidmoves1}
\end{figure}
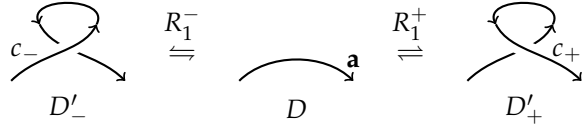 

The complex $C_{\mathcal{F}}(D^\prime_+,R)$ can be identified (as a graded $R$-module) with the complex
\begin{equation}
 C^\bullet_{\mathcal{F}}(D\cup \bigcirc) \oplus C^\bullet_{\mathcal{F}}(D)(-1) \simeq  ( C^\bullet_{\mathcal{F}}(D)\otimes_{R} A )\oplus C^\bullet_{\mathcal{F}}(D)(-1),
\label{eq:identifyingcomplexesR1plus}
\end{equation}
where $(\cdot)$ denotes the (homological) degree shift; that is, given a $\mathbb{Z}^n$ graded module $M^\bullet$, then $\left(M(J)\right)^{I} = M^{I+J}$ for each $I$, $J\in \mathbb{Z}^n$.

Each resolution of $D^\prime_+$ obtained by performing a $0$-resolution of $c_+$ can be identified with a resolution of $D\cup \bigcirc$, while each of the remaining resolutions can be identified with a resolution of $D$. To turn this identification into an isomorphism of $R$-complexes, we identify the complex with the mapping cone of the map $\mathbf{S}_{\mathcal{F}}$ associated to a saddle move. Concretely, we endow the graded $R$-module on the left-hand side of \eqref{eq:identifyingcomplexesR1plus} with the differential
\[ d_\mathcal{F}^\prime = \left( \begin{matrix}
d_{\mathcal{F}}\otimes_{R} id_{A} & 0 \\
\mathbf{S}_{\mathcal{F}}        & d_{\mathcal{F}} 
\end{matrix}\right),\]
where $\mathbf{S}_\mathcal{F}$ is the map associated to a saddle move merging the unknotted component with the circle $\gamma^\prime$ containing \textbf{a}. More explicitly,
\[ \mathbf{S}_{\mathcal{F}}:  C^\bullet_{\mathcal{F}}(D)\otimes_{R} A  \to C^\bullet_{\mathcal{F}}(D):\left( \bigotimes_{\gamma\in\underline{r}} \alpha_{\gamma} \right)\otimes \alpha \mapsto \left( \bigotimes_{\gamma\in\underline{r}\setminus \{ \gamma^\prime \}} \alpha_{\gamma} \right) \otimes m_{\mathcal{F}}(\alpha_{\gamma^\prime}, \alpha).\]
Now, we are ready to define the map associated to the addition of the (positive) curl. This map, denoted by $\Phi_1^+(\mathcal{F})$ or just $\Phi^+_1$, is defined as follows
\begin{align*}
\Phi_{1}^{+}:&\ C^\bullet_{\mathcal{F}}(D) \longrightarrow (C^\bullet_{\mathcal{F}}(D)\otimes_{R} A )\oplus C^\bullet_{\mathcal{F}}(D)(-1) \\
 &\ \bigotimes_{\gamma\in\underline{r}} \alpha_{\gamma} \mapsto \left(\left(\bigotimes_{\gamma\in\underline{r}\setminus \{ \gamma^\prime \}} \alpha_{\gamma}\right) \otimes \left(\alpha_{\gamma^\prime}\otimes t_{\mathcal{F}}(1_{R}) - \Delta_{\mathcal{F}}(\alpha_{\gamma^\prime})\right)\right)\oplus 0
\end{align*}
where $t_{\mathcal{F}}$ is the de-cupped torus map.
To conclude the positive version of the first Reidemeister move, we need to define the map associated to the removal of a positive curl. This map, denoted by $\Psi_1^+(\mathcal{F})$ or simply $\Psi_{1}^{+}$, is given by
\begin{align*}
\Psi_{1}^{+}:&\  (C^\bullet_{\mathcal{F}}(D)\otimes_{R} A )\oplus C^\bullet_{\mathcal{F}}(D)(-1) \longrightarrow C^\bullet_{\mathcal{F}}(D) \\
 &\quad \left(\left( \bigotimes_{\gamma\in\underline{r}} \alpha_{\gamma}\right) \otimes a \right)\ \oplus \bigotimes_{\gamma\in\underline{s}} \delta_{\gamma}\  \mapsto\  \epsilon_{\mathcal{F}}(a) \bigotimes_{\gamma\in\underline{r}} \alpha_{\gamma}.
\end{align*}

Now, let us turn to the negative version of the first Reidemeister move. For our scope it is sufficient to define only the map associated to the creation of a negative curl. Let us denote by $D^\prime_-$ the diagram obtained from $D$ by adding a negative curl on the arc \textbf{a} (see Figure \ref{reidmoves1}). Denote by $c_-$ the crossing of $D^\prime_-$ created by the addition of the curl. 
Similarly to the case of the positive Reidemeister move, there is an identification of the resolutions of $D^\prime_-$ where $c_-$ is replaced with is $0$-resolution and the resolutions of $D$. All the remaining resolutions of $D^\prime_-$ can be identified with the resolutions of $D\cup \bigcirc$. These identifications induce the following isomorphisms of (graded) $R$-modules
\begin{equation}
C_{\mathcal{F}}^\bullet(D^\prime_-) \simeq C_{\mathcal{F}}^\bullet(D)(-1) \oplus C^\bullet_{\mathcal{F}}(D\cup \bigcirc) \simeq   C^\bullet_{\mathcal{F}}(D)(-1)\oplus( C^\bullet_{\mathcal{F}}(D)\otimes_{R} A ). 
\label{eq:identifyingcomplexesR1minus}
\end{equation}

\begin{remark}
Suppose $\mathcal{F}$ is a graded Frobenius algebra. Then the complex $C_\mathcal{F}^\bullet(D,R)$ can be endowed with a second grading (see Subsection \ref{Subs:GradingC_F}). To turn the isomorphisms in \eqref{eq:identifyingcomplexesR1minus} into isomorphisms of \emph{bi-}graded $R$-modules it is necessary to introduce an appropriate quantum degree shift (cf. \cite[Section 6]{BarNatan05cob}). This shift is not necessary in the case of the positive version of the first Reidemeister move.
\end{remark}
As in the case of $R_1^+$, we wish to turn the isomorphisms in \eqref{eq:identifyingcomplexesR1minus} into isomorphisms of chain complexes. In order to do so it is sufficient to endow the rightmost $R$-module in \eqref{eq:identifyingcomplexesR1minus} with the differential
\[ d_\mathcal{F}^\prime = \left( \begin{matrix}
d_{\mathcal{F}} & 0 \\
\mathbf{S}^\prime_\mathcal{F}         & d_{\mathcal{F}} \otimes_{R} id_{A}
\end{matrix}\right);\]
where $\mathbf{S}^\prime_\mathcal{F}$ is the map associated to a saddle move splitting the circle $\gamma^\prime$ containing the arc \textbf{a}. More explicitly,
\[ \mathbf{S}^\prime_\mathcal{F} :  C_{\mathcal{F}}^\bullet(D)  \to C^\bullet_{\mathcal{F}}(D)\otimes_{R} A:\bigotimes_{\gamma\in\underline{r}} \alpha_{\gamma} \mapsto \left(\bigotimes_{\gamma\in\underline{r}\setminus \{ \gamma^\prime \}} \alpha_{\gamma} \right)\otimes \Delta(\alpha_{\gamma^\prime}).\]

\begin{remark}
There is no ambiguity in the labels given to the circles by $\Delta(\alpha_{\gamma^\prime})$ because of the co-commutativity of $\Delta$.
\end{remark}

Finally, we can define the map associated to the addition of a negative curl, denoted by $\Phi_{1}^{-}(\mathcal{F})$ or just $\Phi_{1}^{-}$, as follows
\begin{align*}
\Phi_{1}^{-} :&\ C^\bullet_{\mathcal{F}}(D)\ \longrightarrow\ C^\bullet_{\mathcal{F}}(D)(-1)\oplus (C^\bullet_{\mathcal{F}}(D)\otimes_{R} A ) \\
 &\ \bigotimes_{\gamma\in\underline{r}} \alpha_{\gamma}\ \longmapsto\quad 0\oplus\left(\bigotimes_{\gamma\in\underline{r}} \alpha_{\gamma} \right)\otimes \iota_{\mathcal{F}}(1_{R})
\end{align*}
Now we are finally ready to state (and prove) a result describing the behaviour of $\beta_{\mathcal{F}}(D,R)$ with respect to the maps associated to the first Reidemeister move(s). 

\begin{remark}\label{rem:destabilizzazione_neg}
The map induced by the negative first Reidemeister move has been obtained by composition; the map $\Phi_{1}^{-}$ (resp. its homotopy inverse $\Psi_{1}^{-}$) can be obtained by composing (resp. pre-composing) the map associated to a non-coherent second Reidemeister move (see \cite[Chapter 3, Section 3]{Thesis1} or \cite[Section 4]{BarNatan05cob}) and $\Psi_{1}^{+}$ (resp. $\Phi_{1}^{+}$). With these definitions we have that $\Psi_{1}^{-} \circ \Phi_{1}^{-} = - Id$.
\end{remark}

\begin{proposition}\label{proposition:firstbeta}
Let $D$ be an oriented link diagram. If $D^\prime_{+}$ (resp. $D^\prime_{-}$) is the diagram obtained from $D$ via a positive (resp. negative) first Reidemeister move (with the induced orientation), then
\begin{equation}
\tag{R1p}
\Psi_{1}^{+}(\mathcal{F})(\beta_\mathcal{F}(D)) = \beta_\mathcal{F}(D^\prime_+),\qquad\qquad \Phi_{1}^{+}(\mathcal{F})(\beta_\mathcal{F}(D^\prime_+)) = \beta_\mathcal{F}(D),
\label{eq:invarianceR1pF}
\end{equation}
and
\begin{equation}
\tag{R1n}
(x_1-x_2)(\Phi_{1}^{-})_{*}(\mathcal{F})([\beta_\mathcal{F}(D)]) = -e_{x} [\beta_\mathcal{F}(D^\prime_-)];
\label{eq:invarianceR1nF}
\end{equation}
where $x$ is the label in $\beta_\mathcal{F}(D,R)$ of the circle containing the arc where the first move is performed.
\end{proposition}
\begin{proof}
Let us start from the addition of a positive curl. Suppose $\alpha_{\gamma^\prime} = x \in \{ x_{\circ}, x_{\bullet} \}$. It follows from Equations \eqref{eq:comultandcirclebullet} and \eqref{eq:torusmap} that
\[ \alpha_{\gamma^\prime}\otimes t_{\mathcal{F}}(1_{R_{\mathcal{F}}}) - \Delta_{\mathcal{F}}(\alpha_{\gamma^\prime}) = x \otimes (x + \overline{x}) -  x \otimes x = x \otimes \overline{x},\]
where $\bar{x}$ denotes the conjugation on the set $\{ x_{\circ}, x_{\bullet} \}$.
Identify the oriented resolution of $D^\prime_+$ with the oriented resolution of $D\cup \bigcirc$ as in the definition of $\Phi^+_1$. From the previous considerations it follows that the label of the un-knotted component which does not belong to $D$ in $\Phi_{1}^{+}(\beta_{\mathcal{F}}(D))$ is $\overline{x}$, the label of $\gamma^\prime$ is $x$, and all the other labels remain unchanged. Thus, it follows immediately that
\[ \Phi_{1}^{+}(\beta_{\mathcal{F}}(D)) = \beta_{\mathcal{F}}(D^\prime_+).\]
To conclude the case of the positive $R_{1}$ move, we must verify that $\beta_{\mathcal{F}}(D,R)$ is preserved by $\Psi_1^{+}$. The claim follows from the following facts: (a) if $a = b_{\gamma^\prime}$ then $\epsilon_{\mathcal{F}}(a) = 1$, (b) the direct summand in $C_{\mathcal{F}}^{\bullet}(D^\prime_+)$ corresponding to the oriented resolution of $D^\prime$ is mapped onto the direct summand in $C_{\mathcal{F}}^{\bullet}(D)$ corresponding to the oriented resolution, and (c) the labels on the circles that are not involved in the move and in the circle $\gamma^\prime$ are left invariant by $\Psi_1^{+}$.

Now, let us turn to the behaviour of $\beta_{\mathcal{F}}(D,R)$ with respect to the map associated to the addition of a negative curl. Immediately from the definition it follows that
\begin{equation}
\Phi_{1}^{-}(\beta(D,R)) = \left(\bigotimes_{\gamma\in\underline{r}} b_{\gamma} \right) \otimes \iota(1_{R}) ,
\label{eq:Phi1negative}
\end{equation}
where $\underline{r}$ denotes the oriented resolution of $D$, and the oriented resolution of $D^\prime_-$ is identified with the oriented resolution of $D\cup \bigcirc$. Consider the chain
\[ \eta = 0\oplus\left( \left(\bigotimes_{\gamma\in\underline{r}} b_{\gamma}  \right)\otimes x\right) =0\oplus( \beta_{\mathcal{F}}(D)\otimes x) \]
in $C^\bullet_{\mathcal{F}}(D_{-}^\prime)$. Directly from the definition of $d^\prime_{\mathcal{F}}$ follows
\[ d^\prime_{\mathcal{F}} (\beta_{\mathcal{F}}(D)\oplus 0 )=  \eta. \]
By Equation \eqref{eq:differencebetweentheconjugatesgeneral} we have
\[ \beta_{\mathcal{F}}(D^\prime_-) =  0\oplus\left( \left(\bigotimes_{\gamma\in\underline{r}} b_{\gamma}  \right)\otimes \bar{x}\right) = \eta - e_{x} (x_1 - x_2)\Phi_{1}^{-}(\beta_{\mathcal{F}}(D)). \]
and the claim follows.
\end{proof}

\subsection{Second Reidemeister move}\label{sec:BetainvII}

Let $D$ be an oriented link diagram. Let \textbf{a} and \textbf{b} be two (un-knotted) arcs of $D$ lying in a small ball. Performing a second Reidemeister move on these arcs inserts two adjacent crossings, say $c_1$ and $c_2$, of opposite type (see Figure \ref{reidmoves2}).

\begin{figure}[h]
\centering
\begin{tikzpicture}[scale = .35, thick]

\node at (-1,1.5) {\textbf{a}};
\node at (-1,-1.5) {\textbf{b}};
\draw (2.5,1.5) .. controls +(-1,-1) and +(1,-1) ..  (-2.5,1.5); 
\draw (-2.5,-1.5) .. controls +(1,1) and +(-1,1) ..  (2.5,-1.5);

\node at (4,0) {$\rightleftharpoons$};

\node at (4,1) {$R_2$};
\draw (11.5,1.5) .. controls +(-1,-3) and +(1,-3) ..  (6.5,1.5);
\pgfsetlinewidth{8*\pgflinewidth}
\draw[white] (6.5,-1.5) .. controls +(1,3) and +(-1,3) ..  (11.5,-1.5);
\pgfsetlinewidth{.125*\pgflinewidth} 
\draw (6.5,-1.5) .. controls +(1,3) and +(-1,3) ..  (11.5,-1.5);

\node at (12,0) {$c_2$};
\node at (6,0) {$c_1$};

\node at (9,-2.5) {$D^{\prime\prime}$};
\node at (0,-2.5) {$D$};
\end{tikzpicture}
\caption{The (un-oriented) second Reidemeister move.}
\label{reidmoves2}
\end{figure}
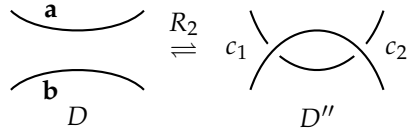

Denote by $D^{\prime\prime}$ the oriented link diagram obtained from $D$ by performing a second Reidemeister move on the arcs \textbf{a} and \textbf{b}. There are four possible resolutions of the pair of crossings $c_1$ and $c_2$. Let $D^{\prime\prime}_{ij}$, with $i,j\in\{ 0,1\}$, be the link obtained from $D^{\prime\prime}$ by performing a $i$-resolution on $c_1$ and a $j$-resolution on $c_2$ (Figure \ref{fig:fourresolutions}).
Notice that there is a natural identification of the link $D_{10}^{\prime\prime}$ with $D$.

\begin{figure}[h]
\centering
\begin{tikzpicture}[scale = .1, thick]
\isplit{-15}{-2}
\osplit{-6}{-2}
\draw (-6,2) .. controls +(-1.5,1.5) and +(1.5,1.5) ..  (-11,2); 
\draw (-11,-2) .. controls +(1.5,-1.5) and +(-1.5,-1.5) ..  (-6,-2);

\draw (-15,2) -- (-16,3); 
\draw (-15,-2) -- (-16,-3);
\draw (-2,2) -- (-1,3); 
\draw (-2,-2) -- (-1,-3);
\node at (-8.5,-7) {$D^{\prime\prime}_{00}$};
\draw[dashed, thin, rounded corners = 1.5] (-16,-4) rectangle (-1,4);

\begin{scope}[shift = {(24,0)}]
\isplit{-15}{-2}
\isplit{-6}{-2}
\draw (-6,2) .. controls +(-1.5,1.5) and +(1.5,1.5) ..  (-11,2); 
\draw (-11,-2) .. controls +(1.5,-1.5) and +(-1.5,-1.5) ..  (-6,-2);

\draw (-15,2) -- (-16,3); 
\draw (-15,-2) -- (-16,-3);
\draw (-2,2) -- (-1,3); 
\draw (-2,-2) -- (-1,-3);
\node at (-8.5,-7) {$D^{\prime\prime}_{01}$};
\draw[dashed, thin, rounded corners = 1.5] (-16,-4) rectangle (-1,4);
\end{scope}

\begin{scope}[shift = {(48, 0)}]
\osplit{-15}{-2}
\osplit{-6}{-2}
\draw (-6,2) .. controls +(-1.5,1.5) and +(1.5,1.5) ..  (-11,2); 
\draw (-11,-2) .. controls +(1.5,-1.5) and +(-1.5,-1.5) ..  (-6,-2);

\draw (-15,2) -- (-16,3); 
\draw (-15,-2) -- (-16,-3);
\draw (-2,2) -- (-1,3); 
\draw (-2,-2) -- (-1,-3);
\node at (-8.5,-7) {$D^{\prime\prime}_{10}$};
\draw[dashed, thin, rounded corners = 1.5] (-16,-4) rectangle (-1,4);
\end{scope}

\begin{scope}[shift = {(72,0)}]
\osplit{-15}{-2}
\isplit{-6}{-2}
\draw (-6,2) .. controls +(-1.5,1.5) and +(1.5,1.5) ..  (-11,2); 
\draw (-11,-2) .. controls +(1.5,-1.5) and +(-1.5,-1.5) ..  (-6,-2);

\draw (-15,2) -- (-16,3); 
\draw (-15,-2) -- (-16,-3);
\draw (-2,2) -- (-1,3); 
\draw (-2,-2) -- (-1,-3);
\node at (-8.5,-7) {$D^{\prime\prime}_{11}$};
\draw[dashed, thin, rounded corners = 1.5] (-16,-4) rectangle (-1,4);
\end{scope}
\end{tikzpicture}
\caption{The possible resolutions of $c_1$ and $c_2$.}
\label{fig:fourresolutions}
\end{figure}
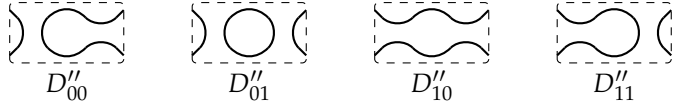

\begin{remark}
Only one among the links $D_{00}^{\prime\prime}$, $D_{10}^{\prime\prime}$, $D_{01}^{\prime\prime}$ and $D_{11}^{\prime\prime}$ inherits the orientation from $D^{\prime\prime}$, and this is either $D_{10}^{\prime\prime}$ or $D_{01}^{\prime\prime}$.
\end{remark}

Similarly to the case of the first Reidemeister move, there is an isomorphism of graded $R$-modules
\begin{equation}
C_\mathcal{F}^{\bullet}(D^{\prime\prime}) \simeq C^{\bullet}_\mathcal{F}(D^{\prime\prime}_{00}) \oplus C^{\bullet}_\mathcal{F}(D^{\prime\prime}_{10})(-1) \oplus C^{\bullet}_\mathcal{F}(D^{\prime\prime}_{01})(-1)\oplus C^{\bullet}_\mathcal{F}(D^{\prime\prime}_{11})(-2).
\label{eq:identifyingcomplexesR2}
\end{equation}
given by the identification of each resolution of $D^{\prime\prime}$ with a resolution of $D^{\prime\prime}_{ij}$ (for a suitable choice of $i$ and $j$). 

\begin{remark}
Assume $\mathcal{F}$ to be a graded Frobenius algebra. To turn the isomorphism in \eqref{eq:identifyingcomplexesR2} into an isomorphism of bi-graded $R$-modules a suitable shift of the quantum degree has to be taken into account (cf. \cite[Section 4]{BarNatan05cob}).
\end{remark}

The isomorphism in \eqref{eq:identifyingcomplexesR2} is not an isomorphism of $R$-complexes. To obtain such an isomorphism it is necessary to modify the differential of the complex on the right-hand-side of \eqref{eq:identifyingcomplexesR2}. This modified differential can be (roughly) defined as follows
\[
d^{\prime\prime}_\mathcal{F}= \begin{pmatrix}
d_\mathcal{F}^{00} & 0 & 0 &  0 \\
\mathbf{S}^{\prime\prime}_{00,10} & d_\mathcal{F}^{10} &  0 & 0\\
\mathbf{S}^{\prime\prime}_{00,01} & 0 & d_\mathcal{F}^{01} &  0 \\
0 & \mathbf{S}^{\prime\prime}_{10,11} & \mathbf{S}^{\prime\prime}_{01,11} & d_\mathcal{F}^{11}  \\
\end{pmatrix}
\]
where $d^{ij}_\mathcal{F}$ is the differential of the complex $C_\mathcal{F}^{\bullet}(D^{\prime\prime}_{ij})$, and
\[\mathbf{S}^{\prime\prime}_{ij,hk}:C_\mathcal{F}^{\bullet}(D^{\prime\prime}_{ij},R) \longrightarrow C_\mathcal{F}^{\bullet}(D^{\prime\prime}_{hk},R) \]
is the map corresponding to a saddle move from $D^{\prime\prime}_{ij}$ to $D^{\prime\prime}_{hk}$. This description is more than sufficient for our scope. The interested reader may consult \cite[Section 5]{Khovanov00} or \cite[Section 4]{BarNatan05cob} for a more detailed description of $d^{\prime\prime}_{\mathcal{F}}$. 

Now, consider the diagram $D^{\prime\prime}_{01}$. Denote by \textbf{c} and \textbf{d} the two arcs appearing in the local picture in Figure \ref{fig:fourresolutions} (see also Figure \ref{fig:arcsinLprimeprime01}). Fix an arc \textbf{g}, meeting $D^{\prime\prime}_{01}$ only at the endpoints, joining \textbf{c} and \textbf{d}. Finally, fix an arc \textbf{e}, meeting $D$ only at the endpoints, joining the arcs \textbf{a} and \textbf{b}.

\begin{figure}[h]
\centering
\begin{tikzpicture}[scale = .15, thick]

\isplit{-15}{-2}
\isplit{-6}{-2}
\draw (-6,2) .. controls +(-1.5,1.5) and +(1.5,1.5) ..  (-11,2); 
\draw (-11,-2) .. controls +(1.5,-1.5) and +(-1.5,-1.5) ..  (-6,-2);

\draw (-15,2) -- (-16,3); 
\draw (-15,-2) -- (-16,-3);
\draw (-2,2) -- (-1,3); 
\draw (-2,-2) -- (-1,-3);
\node at (-8.5,-6.5) {$D^{\prime\prime}_{01}$};
\node at (-4.5,-2.5) {$\gamma^{\prime\prime}$};

\draw[green, dashed] (-15,2) .. controls +(1.5,3.5) and +(-1.5,3.5) ..  (-2,2);
\node at (-8.5,6) {\textbf{g}};

\node at (-2,0) {\textbf{d}};
\node at (-15,0) {\textbf{c}};

\draw (-21,3) .. controls +(-1.5,-1.5) and +(1.5,-1.5) .. (-36,3);
\draw (-21,-3) .. controls +(-1.5,1.5) and +(1.5,1.5) .. (-36,-3);

\node at (-27,3) {\textbf{a}};
\node at (-27,-3) {\textbf{b}};

\node at (-29.5,0) {\textbf{e}};
\draw[green,dashed] (-28.5,-1.8) -- (-28.5,2);
\node at (-28.5,-6.5) {$D = D^{\prime\prime}_{10}$};
\end{tikzpicture}
\caption{}
\label{fig:arcsinLprimeprime01}
\end{figure}

With the notation defined above, and using the notation introduced in Subsection 3.2, we can finally define the map
\[ \Psi_2: C_\mathcal{F}^{\bullet} (D) \longrightarrow C^{\bullet}_\mathcal{F}(D^{\prime\prime}_{00}) \oplus C^{\bullet}_\mathcal{F}(D^{\prime\prime}_{10})(-1) \oplus C^{\bullet}_\mathcal{F}(D^{\prime\prime}_{01})(-1)\oplus C^{\bullet}_\mathcal{F}(D^{\prime\prime}_{11})(-2)\]
as follows
\[ \Psi_2(x) = 0 \oplus x \oplus \left( \mathbf{S}(x,\text{\textbf{e}}) \otimes \iota (1_R ) \right) \oplus 0,\]
where $x$ is an enhanced state, $D$ and $D^{\prime\prime}_{10}$ have been identified, and $\iota(1_R)$ is the label of $\gamma^{\prime\prime}$ (cf. Figure \ref{fig:arcsinLprimeprime01}). Similarly, the up-to-chain-homotopy inverse of $\Psi_2$
\[ \Phi_2:  C^{\bullet}_\mathcal{F}(D^{\prime\prime}_{00}) \oplus C^{\bullet}_\mathcal{F}(D^{\prime\prime}_{10})(-1) \oplus C^{\bullet}_\mathcal{F}(D^{\prime\prime}_{01})(-1)\oplus C^{\bullet}_\mathcal{F}(D^{\prime\prime}_{11})(-2) \longrightarrow C_\mathcal{F}^{\bullet} (D)\]
is given by
\[ \Phi_2 (x_{00}\oplus x_{10} \oplus x_{01} \oplus x_{11}) = x_{10} + \epsilon(x_{\gamma^{\prime\prime}})\mathbf{S}(x_{01},\text{\textbf{g}}),\]
where $x_{ij}$ denotes a (possibly trivial) enhanced state in $C_{\mathcal{F}}^{\bullet}(D^{\prime\prime}_{ij})$, and $x_{\gamma^{\prime\prime}}$ denotes the label of $\gamma^{\prime\prime}$ in $x_{01}$ (cf. Figure \ref{fig:arcsinLprimeprime01}).

Before stating the results concerning $\beta_{\mathcal{F}}(D,R)$ recall that a $R_2$ move is coherent if the arcs \textbf{a} and \textbf{b} involved are oriented as in the right of Figure \ref{fig:coherentR2}.

\begin{figure}[h]
\centering
\begin{tikzpicture}[scale = .1, thick, xscale = -1]

\ocross{-15}{-2}
\icross{-6}{-2}
\draw (-6,2) .. controls +(-1.5,1.5) and +(1.5,1.5) ..  (-11,2); 
\draw (-11,-2) .. controls +(1.5,-1.5) and +(-1.5,-1.5) ..  (-6,-2);

\draw[->] (-15,2) -- (-16,3); 
\draw[->] (-15,-2) -- (-16,-3);
\draw (-2,2) -- (-1,3); 
\draw (-2,-2) -- (-1,-3);

\node at (3,0) {$\rightleftharpoons$};

\draw[->] (21,3) .. controls +(-2,-2) and +(2,-2) ..  (8,3); 
\draw[<-] (8,-3) .. controls +(2,2) and +(-2,2) ..  (21,-3);

\begin{scope}[shift = {+(60,0)}]

\ocross{-15}{-2}
\icross{-6}{-2}
\draw (-6,2) .. controls +(-1.5,1.5) and +(1.5,1.5) ..  (-11,2); 
\draw (-11,-2) .. controls +(1.5,-1.5) and +(-1.5,-1.5) ..  (-6,-2);

\draw[->] (-15,2) -- (-16,3); 
\draw (-15,-2) -- (-16,-3);
\draw (-2,2) -- (-1,3); 
\draw[->] (-2,-2) -- (-1,-3);

\node at (3,0) {$\rightleftharpoons$};

\draw[->] (21,3) .. controls +(-2,-2) and +(2,-2) ..  (8,3); 
\draw[->] (8,-3) .. controls +(2,2) and +(-2,2) ..  (21,-3);
\end{scope}
\end{tikzpicture}
\caption{Non-coherent (left) and coherent (right) versions of the second Reidemeister move.}
\label{fig:coherentR2}
\end{figure}
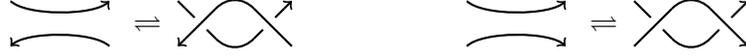

\begin{proposition}\label{proposition:secondcoerbeta}
Let $D$ be an oriented link diagram. Let $D^{\prime\prime}$ be the oriented link diagram obtained from $D$ via a coherent second Reidemeister move. Then
\begin{equation}
\tag{R2c}
\Psi_{2}(\beta_{\mathcal{F}}(D,R)) = \beta_{\mathcal{F}}(D^{\prime\prime},R)\quad \text{and}\quad \Phi_{2}(\beta_{\mathcal{F}}(D^{\prime\prime},R)) = \beta_{\mathcal{F}}(D,R).
\label{eq:invarianceR2c}
\end{equation}
\end{proposition}
\begin{proof}
Throughout this proof we will keep the notation shown in Figure \ref{fig:arcsinLprimeprime01}. Let $\underline{r}$ be the oriented resolution of $D$. 
First, let us investigate the behaviour of $\beta_{\mathcal{F}}(D,R)$ with respect to the map $\Psi_2$. It is a simple consequence of the Jordan curve theorem that if the move is coherent then \textbf{a} and \textbf{b} do not belong to the same circle in $\underline{r}$. Let $\gamma_{\text{\textbf{a}}}$ and $\gamma_{\text{\textbf{b}}}$ be the circles to which \textbf{a} and \textbf{b}, respectively, belong to. It follows directly from the definitions that
\[ \Psi_2(\beta_{\mathcal{F}}(D)) = 0 \oplus \beta_{\mathcal{F}}(D) \oplus \left( \left( \bigotimes_{\gamma\in \underline{r}\setminus \{ \gamma_{\text{\textbf{a}}},\gamma_{\text{\textbf{b}}}\}} b_\gamma \right)\otimes m(b_{\gamma_{\text{\textbf{a}}}},b_{\gamma_{\text{\textbf{b}}}}) \otimes \iota(1_{R}) \right) \oplus 0.\]
As the move is coherent the labels in $\beta_{\mathcal{F}}(D,R)$ of $\gamma_{\text{\textbf{a}}}$ and of $\gamma_{\text{\textbf{b}}}$ are conjugate. Thus, by Equation \eqref{eq:productandcirclebullet2} we have
\[ m(b_{\gamma_{\text{\textbf{a}}}},b_{\gamma_{\text{\textbf{b}}}}) = m(b_{\gamma_{\text{\textbf{a}}}},\overline{b_{\gamma_{\text{\textbf{a}}}}}) =0. \]
Another consequence of the coherence of the move is that the oriented resolution of $D^{\prime\prime}$ is identified (via the isomorphism in \eqref{eq:identifyingcomplexesR2}) with the oriented resolution of $D^{\prime\prime}_{10}$. Thus, it follows that
\[ \Psi_2(\beta_{\mathcal{F}}(D,R)) = \beta_{\mathcal{F}}(D^{\prime\prime},R).\]
As we argued before, the isomorphism in \eqref{eq:identifyingcomplexesR2} sends $\beta_{\mathcal{F}}(D^{\prime\prime},R)$ to
\[ 0 \oplus \beta_{\mathcal{F}}(D,R) \oplus 0 \oplus 0. \]
With the same reasoning as above, from the coherence of the move it follows that
\[\gamma_{\text{\textbf{a}}} \ne \gamma_{\text{\textbf{b}}}\quad \text{and}\quad b_{\gamma_{\text{\textbf{b}}}} = \overline{b_{\gamma_{\text{\textbf{a}}}}}. \]
From Equation \eqref{eq:productandcirclebullet1}, and from the considerations just made, we obtain
\[ \mathbf{S}(\beta_{\mathcal{F}}(D,R),\text{\textbf{g}}) = 0. \]
Since
\[\Phi_2 (0 \oplus \beta_{\mathcal{F}}(D,R) \oplus 0 \oplus 0 ) = \beta(D,R) + \mathbf{S}(\beta_{\mathcal{F}}(D,R),\text{\textbf{g}})\]
the claim follows.
\end{proof}

\subsection{Third Reidemeister move}
We have arrived at the case of the third Reidemeister move. 
This move is the hardest to deal with because it comes in several versions. Moreover, the number of crossings is equal across both sides of the move, so there is no loss of complexity.

We shall avoid giving an explicit description of these maps,  and instead describe the procedure used to associate a map to each version of the third Reidemeister move.

It is necessary to remark that the set of all versions of third Reidemeister moves can be seen as generated by a sequence of (coherent versions of the) second Reidemeister moves and the moves in Figure \ref{Fig:generatorsR3} (see \cite[Lemma 2.6]{Polyak10}). 

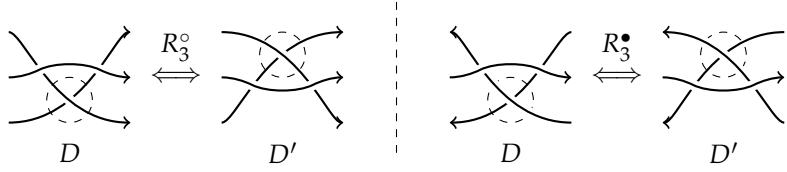
\begin{figure}[h]
\centering
\begin{tikzpicture}[scale =.2]
\draw[thick,->] (-4,-3)  .. controls +(5,0)  and +(-1,0) ..(4,3);

\pgfsetlinewidth{10*\pgflinewidth}
\draw[white] (-4,3)  .. controls +(1,0)  and +(-5,0) ..(4,-3);
\pgfsetlinewidth{.1*\pgflinewidth}
\draw[thick,->] (-4,3)  .. controls +(1,0)  and +(-5,0) ..(4,-3);
\pgfsetlinewidth{10*\pgflinewidth}
\draw[white] (-4,0)  .. controls +(.75,0) and +(-.5,-0.25) .. (-2,0.5) .. controls +(1,.5) and +(-1,0.5) .. (2,0.5) .. controls +(.5,-0.25) and +(-.75,0) .. (4,0);
\pgfsetlinewidth{.1*\pgflinewidth}
\draw[thick,->]  (-4,0)  .. controls +(.75,0) and +(-.5,-0.25) .. (-2,0.5) .. controls +(1,.5) and +(-1,0.5) .. (2,0.5) .. controls +(.5,-0.25) and +(-.75,0) .. (4,0);
\draw[dashed] (0,-1.5) circle (1.5);

\begin{scope}[shift={+(14,0)}]
\draw[thick,->] (-4,-3)  .. controls +(1,0)  and +(-5,0) ..(4,3);

\pgfsetlinewidth{10*\pgflinewidth}
\draw[white](-4,3)  .. controls +(5,0)  and +(-1,0) ..(4,-3);
\pgfsetlinewidth{.1*\pgflinewidth}
\draw[thick,->] (-4,3)  .. controls +(5,0)  and +(-1,0) ..(4,-3);
\draw[dashed] (0,1.5) circle (1.5);
\pgfsetlinewidth{10*\pgflinewidth}
\draw[white]  (-4,0)  .. controls +(.75,0) and +(-.5,0.25) .. (-2,-0.5) .. controls +(1,-.5) and +(-1,-0.5) .. (2,-0.5) .. controls +(.5,0.25) and +(-.75,0) .. (4,0);
\pgfsetlinewidth{.1*\pgflinewidth}
\draw[thick,->]  (-4,0)  .. controls +(.75,0) and +(-.5,0.25) .. (-2,-0.5) .. controls +(1,-.5) and +(-1,-0.5) .. (2,-0.5) .. controls +(.5,0.25) and +(-.75,0) .. (4,0);
\end{scope}

\begin{scope}[shift={+(29,0)}]
\draw[thick,<-] (-4,-3)  .. controls +(5,0)  and +(-1,0) ..(4,3);
\draw[dashed] (0,-1.5) circle (1.5);
\pgfsetlinewidth{10*\pgflinewidth}
\draw[white] (-4,3)  .. controls +(1,0)  and +(-5,0) ..(4,-3);
\pgfsetlinewidth{.1*\pgflinewidth}
\draw[thick,<-] (-4,3)  .. controls +(1,0)  and +(-5,0) ..(4,-3);
\pgfsetlinewidth{10*\pgflinewidth}
\draw[white] (-4,0)  .. controls +(.75,0) and +(-.5,-0.25) .. (-2,0.5) .. controls +(1,.5) and +(-1,0.5) .. (2,0.5) .. controls +(.5,-0.25) and +(-.75,0) .. (4,0);
\pgfsetlinewidth{.1*\pgflinewidth}
\draw[thick,->]  (-4,0)  .. controls +(.75,0) and +(-.5,-0.25) .. (-2,0.5) .. controls +(1,.5) and +(-1,0.5) .. (2,0.5) .. controls +(.5,-0.25) and +(-.75,0) .. (4,0);
\end{scope}

\begin{scope}[shift={+(43,0)}]
\draw[dashed] (0,1.5) circle (1.5);
\draw[thick,<-] (-4,-3)  .. controls +(1,0)  and +(-5,0) ..(4,3);
\pgfsetlinewidth{10*\pgflinewidth}
\draw[white]  (-4,3)  .. controls +(5,0)  and +(-1,0) ..(4,-3);
\pgfsetlinewidth{.1*\pgflinewidth}
\draw[thick,<-] (-4,3)  .. controls +(5,0)  and +(-1,0) ..(4,-3);

\pgfsetlinewidth{10*\pgflinewidth}
\draw[white]  (-4,0)  .. controls +(.75,0) and +(-.5,0.25) .. (-2,-0.5) .. controls +(1,-.5) and +(-1,-0.5) .. (2,-0.5) .. controls +(.5,0.25) and +(-.75,0) .. (4,0);
\pgfsetlinewidth{.1*\pgflinewidth}
\draw[thick,->]  (-4,0)  .. controls +(.75,0) and +(-.5,0.25) .. (-2,-0.5) .. controls +(1,-.5) and +(-1,-0.5) .. (2,-0.5) .. controls +(.5,0.25) and +(-.75,0) .. (4,0);
\end{scope}
\draw[dashed, thin] (21.5,5) -- (21.5,-5);

\node at (36,0) {$\Longleftrightarrow$};
\node at (7,0) {$\Longleftrightarrow$};
\node at (7,2) {$R_{3}^{\circ}$};
\node at (36,2) {$R_{3}^{\bullet}$};
\node at (0,-5) {$D$};
\node at (14,-5) {$D^\prime$};
\node at (29,-5) {$D$};
\node at (43,-5) {$D^\prime$};
\end{tikzpicture}
\caption{Two version of the third Reidemeister move.}\label{Fig:generatorsR3}
\end{figure}

Recall that a third Reidemeister move is braid-like (or coherent) if it can be realized as a relation in the braid group. All braid-like third Reidemeister moves can be obtained from the $R_{3}^{\circ}$ move via a sequence of coherent $R_2$ moves (\cite[Lemma 2.6]{Polyak10}). So, it is sufficient to prove the invariance of the $\beta_{\mathcal{F}}$-cycles with respect to $R_{3}^{\circ}$ move. 

In order to define the map associated to $R^\circ_3$ and $R^\bullet_3$ we make use of the so-called categorified Kauffman trick (\cite[Section 4]{BarNatan05cob}). All the maps associated to the other third Reidemeister moves will be defined as a composition. Since we are concerned only with braid-like moves, we shall describe explicitly only the map associated of the $R_{3}^\circ$ move. The map associated to the $R_{3}^\bullet$ move can be described similarly.

First, write the complexes associated to both sides of the Reidemeister move as cones. More precisely,
\[ C_{\mathcal{F}}(D) = \text{Cone}\left( \mathbf{S}: C_{\mathcal{F}}(D_{0})\to C_{\mathcal{F}}(D_{1})\right),\]
where $D_{i}$ is the diagram obtained by performing the $i$-resolution on the crossing highlighted in Figure \ref{Fig:generatorsR3}, and $\mathbf{S}$ is the map associated to the saddle connecting the diagrams $D_{0}$ and $D_{1}$. An analogous reasoning works for $D^\prime$.

\begin{remark}
One can define the maps associated to the braid-like third Reidemeister moves directly using the categorified Kauffman trick, instead of defining them by composition. These maps may differ from those obtained by composition. Nonetheless, Lemma \ref{lemma:techthird} applies almost verbatim, and one obtains that the maps defined via the categorified Kauffman trick still preserve the $\beta_{\mathcal{F}}$-cycles.
\end{remark}

One notices that the links $D_{0}$ and $D^\prime_{0}$ (with the obvious notation) are related to the link $D^{\prime \prime}$ (Figure \ref{Fig:terzaintermedia}) by a coherent $R_{2}$. This implies that there are maps induced by the two $R_{2}$ moves,
\[ f: C_{\mathcal{F}}(D_{0})\to C_{\mathcal{F}}(D^{\prime\prime}) \quad f^\prime: C_{\mathcal{F}}(D^\prime_{0})\to C_{\mathcal{F}}(D^{\prime\prime}),\]
which are quasi-isomorphisms. Denote the respective up-to-homotopy inverses by $g$ and $g^\prime$.

\begin{figure}[h]
\centering
\begin{tikzpicture}[scale =.2]
\draw[thick,->] (-4,-3) .. controls +(2,0)  and +(-2,0) .. (0,-2) .. controls +(2,0)  and +(-2,0) ..(4,-3);

\pgfsetlinewidth{10*\pgflinewidth}
\draw[white] (-4,3) .. controls +(2,0)  and +(-2,0) .. (0,-1) .. controls +(2,0)  and +(-2,0) ..(4,3);
\pgfsetlinewidth{.1*\pgflinewidth}
\draw[thick,->] (-4,3) .. controls +(2,0)  and +(-2,0) .. (0,-1) .. controls +(2,0)  and +(-2,0) ..(4,3);
\pgfsetlinewidth{10*\pgflinewidth}
\draw[white] (-4,0)  .. controls +(.75,0) and +(-.5,-0.25) .. (-2,0.5) .. controls +(1,.5) and +(-1,0.5) .. (2,0.5) .. controls +(.5,-0.25) and +(-.75,0) .. (4,0);
\pgfsetlinewidth{.1*\pgflinewidth}
\draw[thick,->]  (-4,0)  .. controls +(.75,0) and +(-.5,-0.25) .. (-2,0.5) .. controls +(1,.5) and +(-1,0.5) .. (2,0.5) .. controls +(.5,-0.25) and +(-.75,0) .. (4,0);

\begin{scope}[shift={+(14,0)}]
\draw[thick,->] (-4,3) -- (4,3);

\draw[thick,->] (-4,-3)--(4,-3);

\draw[thick,->]  (-4,0) -- (4,0);
\end{scope}

\begin{scope}[shift={+(28,0)}]
\draw[thick,->] (-4,3) .. controls +(2,0)  and +(-2,0) .. (0,2) .. controls +(2,0)  and +(-2,0) ..(4,3);

\pgfsetlinewidth{10*\pgflinewidth}
\draw[white] (-4,-3) .. controls +(2,0)  and +(-2,0) .. (0,1) .. controls +(2,0)  and +(-2,0) ..(4,-3);
\pgfsetlinewidth{.1*\pgflinewidth}
\draw[thick,->] (-4,-3) .. controls +(2,0)  and +(-2,0) .. (0,1) .. controls +(2,0)  and +(-2,0) ..(4,-3);
\pgfsetlinewidth{10*\pgflinewidth}
\draw[white] (-4,0)  .. controls +(.75,0) and +(-.5,0.25) .. (-2,-0.5) .. controls +(1,-.5) and +(-1,-0.5) .. (2,-0.5) .. controls +(.5,0.25) and +(-.75,0) .. (4,0);
\pgfsetlinewidth{.1*\pgflinewidth}
\draw[thick,->]  (-4,0)  .. controls +(.75,0) and +(-.5,0.25) .. (-2,-0.5) .. controls +(1,-.5) and +(-1,-0.5) .. (2,-0.5) .. controls +(.5,0.25) and +(-.75,0) .. (4,0);
\end{scope}
\node at (0,-5) {$D_{0}$};
\node at (14,-5) {$D^{\prime\prime}$};
\node at (28,-5) {$D^\prime_{0}$};
\end{tikzpicture}
\caption{The links $D_{0}$, $D^{\prime\prime}$ and $D_{0}^\prime$.}\label{Fig:terzaintermedia}
\end{figure}
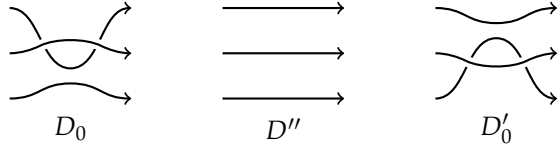
The main point is that these maps are respectively a strong deformation retract and an inclusion in a deformation retract (\cite[Definition 4.3]{BarNatan05cob}). Hence, by \cite[Lemma 4.5]{BarNatan05cob} we have the quasi-isomorphisms
\[\Phi:\: C_{\mathcal{F}}(D) =\text{Cone}\left( \mathbf{S}: C_{\mathcal{F}}(D_{0})\to C_{\mathcal{F}}(D_{1})\right) \rightarrow \text{Cone}\left( \mathbf{S} \circ g : C_{\mathcal{F}}(D^{\prime\prime})\to C_{\mathcal{F}}(D_{1})\right)\]
and
\[\Phi^\prime:\: C_{\mathcal{F}}(D^\prime) =\text{Cone}\left( \mathbf{S} : C_{\mathcal{F}}(D^\prime_{0})\to C_{\mathcal{F}}(D^\prime_{1})\right) \rightarrow \text{Cone}\left( \mathbf{S} \circ g^\prime: C_{\mathcal{F}}(D^{\prime\prime})\to C_{\mathcal{F}}(D_{1})\right),\]
as well as their up to homotopy inverses $\Psi$ and $\Psi^\prime$. Moreover, these maps can be explicitly computed in terms of $f$, $f^\prime$ and their up-to-homotopy inverses $g$ and $g^\prime$.
Finally, one notices that the cones over $ \mathbf{S} \circ g$ and over $ \mathbf{S} \circ g^\prime$ can be identified. Using this identification, the maps associated to the $R_{3}^{\circ}$ can be defined as follows
\[ \Psi_{3} = \Phi \circ \Psi^\prime,\quad \text{and}\quad \Phi_{3} =  \Phi^\prime \circ \Psi. \]
The key point of the invariance of the $\beta_{\mathcal{F}}$-cycles is the following (technical) lemma which is left as an exercise (or see \cite[Proposition 3.15]{Thesis1}).

\begin{lemma}\label{lemma:techthird}
Given a cone over a chain map $S$
\[ \Gamma = \text{Cone}(S: C \longrightarrow D)\]
and an inclusion in a deformation retract (resp. a strong deformation retract)
\[ f: C \to C^\prime\qquad \text{( resp. }g: C^\prime \to C \text{ ),} \]
denote by $F$ (resp. $G$) the quasi-isomorphism
\[ F: \Gamma  \longrightarrow  \textrm{Cone}(S \circ g: C^\prime \longrightarrow D)\qquad \text{(resp.}
\ G:   \textrm{Cone}(S \circ g: C^\prime \longrightarrow D) \longrightarrow \Gamma\text{ )}\]
induced by $f$ (resp. $g$). If $f(x) = x^\prime$ and $ g(x^\prime) = x$, then
\[ F(x \oplus 0) = x^\prime \oplus 0\quad\text{and}\quad G(x^\prime \oplus 0) = x \oplus 0\]
\end{lemma}

As a consequence of the previous lemma we obtain the following proposition.

\begin{proposition}\label{proposition:betathird}
Let $D$ and $D^\prime$ be two oriented link diagrams related by a coherent third Reidemeister move. Then
\begin{equation}
\tag{R3c}
\Psi_3 (\beta_{\mathcal{F}}(D,R)) = \beta_{\mathcal{F}}(D^\prime,R)\quad\text{and}\quad \Phi_3 (\beta_{\mathcal{F}}(D^\prime,R)) = \beta_{\mathcal{F}}(D,R).
\label{eq:invarianceR3c}
\end{equation}
\end{proposition}
\begin{proof}
The claim is an immediate consequence of Lemma \ref{lemma:techthird} and Proposition \ref{proposition:secondcoerbeta}.
\end{proof}

\subsection{The $\beta_{\mathcal{F}}$-invariants}\label{Subs:betaFinvariants}

Recall that, as said in the introduction, the map associated to a Markov move is the map associated to the corresponding Reidemeister move on the braid closure. Now we can prove Theorem \ref{theorem:main1}.

\begin{proof}[of Theorem \ref{theorem:main1}]
Every sequence of transverse Markov moves translates into a sequence of positive first Reidemeister moves, and braid-like second and third Reidemeister moves. A negative stabilization translates into a negative first Reidemeister move. The theorem follows immediately from Propositions \ref{proposition:firstbeta}, \ref{proposition:secondcoerbeta} and \ref{proposition:betathird}.
\end{proof}

\begin{definition}
Given a braid diagram $B$, the \emph{$\beta_{\mathcal{F}}$-invariants associated to $B$} are the cycles $\beta_{\mathcal{F}}(B)= \beta_{\mathcal{F}}(\widehat{B},R)$ and $\overline{\beta}_{\mathcal{F}}(B) = \overline{\beta}_{\mathcal{F}}(\widehat{B},R)$.
\end{definition}

\begin{remark}
Since the $\beta_{\mathcal{F}}$-cycles are defined for all oriented link diagrams, one may wonder if the $\beta_{\mathcal{F}}$-cycles define invariants for other representations of transverse links. For example, one may ask if the $\beta_{\mathcal{F}}$-cycles define invariants for transverse front projections (see \cite{Etnyre05} for a definition). However, this is not true. The reader may consult \cite[Chapter 4, Subsection 2.3]{Thesis1}.
\end{remark}

Let $B$ and $B^\prime$ be two braids representing the transverse links $T$ and $T^\prime$, respectively. Given a sequence $\Sigma$ of Markov (resp. Reidemeister) moves from $B$ to $B^\prime$ (resp. from $\widehat{B}$ to $\widehat{B^\prime}$), denote by
\[\Phi_{\Sigma}: C^\bullet_{\mathcal{F}}(\widehat{B},R) \longrightarrow C^\bullet_{\mathcal{F}}(\widehat{B^\prime},R)\]
the map associated to $\Sigma$. If $B$ and $B^\prime$ represent the same transverse link, then $\Sigma$ can be taken to be a sequence of transverse Markov moves. It follows that $\Phi_{\Sigma}$ sends the pair $(\beta_{\mathcal{F}}(B), \overline{\beta}_{\mathcal{F}}(B))$ to the pair $(\beta_{\mathcal{F}}(B^\prime), \overline{\beta}_{\mathcal{F}}(B^\prime))$.

\begin{remark}
Not all the sequences of Reidemeister moves between two braid closures can be obtained as composition of Markov moves. In particular, the set of maps associated to sequences of Markov moves is \textit{a priori} different from the set of maps associated to sequences of Reidemeister moves.
\end{remark}

Suppose that $T$ and $T^\prime$ are two distinct transverse knots. We will say that $T$ and $T^\prime$ are \emph{weakly} (resp. \emph{strongly}) \emph{distinguished} by the $\beta_{\mathcal{F}}$-invariants if it does \underline{not} exist a sequence $\Sigma$ of Markov (resp. Reidemeister) moves such that 
\[\Phi_{\Sigma}(\beta_{\mathcal{F}}(B)) = \beta_{\mathcal{F}}(B^\prime),\quad \text{and}\quad \Phi_{\Sigma}(\overline{\beta}_{\mathcal{F}}(B)) = \overline{\beta}_{\mathcal{F}}(B^\prime).\]
If the elements of a non-simple pair are weakly (resp. strongly) distinguished by the $\beta_{\mathcal{F}}$-invariants we will say that the $\beta_{\mathcal{F}}$-invariants are \emph{weakly} (resp. \emph{strongly}) \emph{effective}.

\begin{remark}
We do not know whether the strong effectiveness implies the weak effectiveness or not. This is related to the study of the maps induced by the Markov moves in a Khovanov-type theory. It goes beyond the scope of the present paper to explore this subject.
\end{remark}

\begin{remark}
We expect the strong (resp. weak) effectiveness to depend on the Frobenius algebra and on the base ring $R$. Since we do not have the means to study the effectiveness of the $\beta_{\mathcal{F}}$-invariants in full generality, this subject will not be explored further in this paper. 
\end{remark}

In Theorem \ref{theorem:main1} it is stated that the $\beta_{\mathcal{F}}$-invariants are preserved by  the maps induced by a sequence of transverse Markov moves. Comparing it with Proposition \ref{prop:LipshitzNgSarkar} one may notice that the signs of the $\beta_{\mathcal{F}}$-invariants are also preserved. However, for each oriented link diagram $D$, the map $-Id_{C_{\mathcal{F}}^\bullet(D,R)}$ can be always realised as the map associated to a sequence of Reidemeister moves from $D$ to itself (see \cite{BarNatan05cob, Jacobsson04}). Moreover, with our choices for the maps associated to a generating set of Reidemeister moves, $-Id$ can be realised with a sequence of Markov moves (see Remark \ref{rem:destabilizzazione_neg}).
So, if there is a sequence $\Sigma$ of Markov (resp. Reidemeister) moves, from $B$ to $B^\prime$ (resp. from $\widehat{B}$ to $\widehat{B^\prime}$), such that
\[\Phi_\Sigma (\beta_{\mathcal{F}}(B^\prime,R)) = -\beta_{\mathcal{F}}(B^\prime,R),\quad \text{and}\quad \Phi_{\Sigma}( \overline{\beta}_{\mathcal{F}}(B^\prime,R)) = -\overline{\beta}_{\mathcal{F}}(B^\prime,R), \]
then there is also a sequence of Markov (resp. Reidemeister) moves $\Sigma^\prime$, from $B$ to $B^\prime$ (resp. from $\widehat{B}$ to $\widehat{B^\prime}$), such that
\[\Phi_{\Sigma^\prime} (\beta_{\mathcal{F}}(B^\prime)) = \beta_{\mathcal{F}}(B^\prime),\quad \text{and}\quad \Phi_{\Sigma^\prime}( \overline{\beta}_{\mathcal{F}}(B^\prime)) = \overline{\beta}_{\mathcal{F}}(B^\prime). \]

\begin{remark}
More in general, there is an action of the group of the chain homotopy equivalences induced by sequences of  Reidemeister (resp. Markov) moves from $\widehat{B}$ to itself, on $C_{\mathcal{F}}^\bullet (\widehat{B},R)$. Following \cite{Jacobsson04}, we call this group the \emph{monodromy group} (resp. the \emph{braid monodromy group}) and denote it by $M_{\widehat{B}}$ (resp. $BM_{B}$). It has been shown in \cite{Jacobsson04} that the monodromy group contains more than just the identity and its opposite. Let $\Sigma$ be a sequence of Reidemeister (resp. Markov) moves from $B$ to $B^\prime$. If $\Phi_{\Sigma}$ sends the $\beta_{\mathcal{F}}$-invariants of $B$ to a pair of elements in the $M_{\widehat{B^\prime}}$-orbit (resp. $BM_{B^\prime}$-orbit) of the $\beta_{\mathcal{F}}$-invariants of $B^\prime$, then the corresponding transverse links cannot be strongly (resp. weakly) distinguished by the $\beta_{\mathcal{F}}$-invariants.
\end{remark}

At this point the following question arises naturally; how to use the $\beta_{\mathcal{F}}$-invariants to distinguish two transverse links? In other words, how can we prove that two cycles, belonging to different chain complexes, are not mapped one onto the other by a given set of chain homotopy equivalences? An approach is to look at their homology classes. For example, if one of the cycles has trivial homology class while the other does not, then the two cycles cannot be mapped one onto the other. Another approach is to use some extra structure on the complex (e.g. a filtration) which is preserved by the given set of chain homotopies. The first approach was originally used by Plamenevskaya, while the latter approach was used by Lipshitz, Ng and Sarkar. In the next section, we shall make use of the $\mathbb{F}[U]$-module structure of Bar-Natan's homology  to extract some information from the $\beta_{BN}$-invariants. 

To conclude this section we shall prove a ``uniqueness'' result for the $\beta_{\mathcal{F}}$-invariants (i.e. Theorem \ref{theorem:main2}). 

\subsection{A uniqueness property}

Assume $R = R_{\mathcal{F}}$ to be a unique factorization domain (UFD). Even though not every Frobenius algebra which belongs to the family described in Subsection \ref{Sec:exFA} satisfies this hypothesis, a large class of them does. In particular, this hypothesis is satisfied by the algebras $BIG$, $BN$, $VT$, $Kh$, $OLee$ and $TLee$.

Suppose that we have a way to assign to each oriented link diagram $D$ an enhanced state $x(D)\in C_{\mathcal{F}}^{\bullet}(D,R)$, whose underlying resolution is the oriented one. Recall that given a circle $\gamma$ in the oriented resolution of $D$, $b_\gamma$ (resp. $\bar{b}_{\gamma}$) denotes the label of $\gamma$ in  $\beta_{\mathcal{F}}(D,R)$ (resp. $\overline{\beta}_{\mathcal{F}}(D,R)$).

\begin{lemma}\label{lemma:keylemmabetaunique}
Let $D$ be an oriented link diagram. Denote by \textbf{a} and \textbf{b} two unknotted arcs of $D$ as in Figure \ref{fig:relativeorientationaandb}, and by $D^\prime$ the link diagram obtained by performing a coherent second Reidemeister move along \textbf{a} and \textbf{b}. Finally, denote by $\gamma_{\text{\textbf{a}}}$ and $\gamma_{\text{\textbf{b}}}$ are the circles int the oriented resolution of $D$ containing \textbf{a} and \textbf{b}, respectively. Suppose that
\[ \Phi_2(x(D)) = x(D^\prime)\quad \text{and}\quad \Psi_{2}(x(D^\prime)) = x(D),\]
where $\Psi_{2}$ and $\Phi_{2}$ are the maps associated to the second Reidemeister move and its inverse, respectively. Then, the labels of the circles $\gamma_{\text{\textbf{a}}}$ and $\gamma_{\text{\textbf{b}}}$ in $x(D)$ are, respectively, ($R$-)multiples of either $b_{\gamma_{\text{\textbf{a}}}}$ and $b_{\gamma_{\text{\textbf{b}}}}$, or of $\bar{b}_{\gamma_{\text{\textbf{a}}}}$ and $\bar{b}_{\gamma_{\text{\textbf{b}}}}$.
\end{lemma}
\begin{figure}
\centering
\begin{tikzpicture}[scale = .35]
\draw[dashed] (0,0) circle (1.415);
\draw[thick, ->] (-1,-1) .. controls +(.5,.5) and +(.5,-.5) .. (-1,1);
\draw[thick,  ->] (1,-1) .. controls +(-.5,.5) and +(-.5,-.5) .. (1,1);
\draw[fill, white] (-1.2,-.5) rectangle (-1.6,.5);
\node at (-1.415,0) {\textbf{a}};
\draw[fill, white] (1.2,-.5) rectangle (1.6,.5);
\node at (1.415,0) {\textbf{b}};
\end{tikzpicture}
\caption{Two coherently oriented arcs.}\label{fig:relativeorientationaandb}
\end{figure}
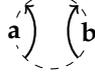
\begin{proof}
Denote by $\underline{r}$ and $\underline{r}^\prime$ the oriented resolutions of $D$ and $D^\prime$, respectively. Finally, denote by $\underline{s}$ the resolution of $D^\prime$ where all crossings but the two added by the second Reidemeister move are resolved as in the oriented resolution.

Let $a$ and $b$ the labels of the circles $\gamma_{\text{\textbf{a}}}$ and $\gamma_{\text{\textbf{b}}}$ in $x(D)$. 
Since $\Psi_2(x(D)) = x(D^\prime) \in A_{\underline{r}^\prime}$, it is immediate that $m(a,b) = 0$. Thus, $a$ and $b$ must be zero divisors in $A$. It follows (since $R[X]$ is an UFD) that $a$ and $b$ belong to either the ideal generated by $x_\circ$ or to the ideal generated by $x_\bullet$ in $A_{BN}$. Moreover, the two labels should belong to different ideals. Since $b_{\gamma_{\text{\textbf{a}}}}$ is either $x_\circ$ or $x_\bullet$ and $b_{\gamma_{\text{\textbf{b}}}} = \bar{b}_{\gamma_{\text{\textbf{a}}}}$ the claim follows.
\end{proof}

\begin{lemma}\label{lemma:uniquebeta}
Let $D$ be a non-split oriented link diagram (i.e.\: $D$ is connected as a planar graph). If $x(D)$ is invariant under coherent Reidemeister moves of the second type, then either $x(D) = r\beta(D,R)$ or $x(D) = r\bar{\beta}(D,R)$, for some $r\in R$.
\end{lemma}
\begin{proof}
If two circles in the oriented resolution share a crossing it is possible to perform a coherent $R_2$ involving those circles. 
Thus, by Lemma \ref{lemma:keylemmabetaunique} each pair of circles sharing a crossing should be labeled either as in $\beta$ or $\bar{\beta}$, up to the multiplication by an element of $R$. Since $D$ has only one split component, the Seifert graph is connected. So, if the label of a single circle is chosen, all the other labels are determined up to multiplication by an element of $R$, and the claim follows.
\end{proof}

Let $D$ be an oriented link diagram, and let $D_1, ..., D_k$ be its split components (i.e. the connected components of $D$ seen as a $4$-valent graph). We will say that $D_{i}$ and $D_{j}$ have \emph{compatible orientations} if there exists ball $\mathbb{B}$ intersecting $D$ in two unknotted arcs \textbf{a} and \textbf{b}, with \textbf{a} belonging to $D_i$ and \textbf{b} belonging to $D_j$, which is ambient isotopic in $\mathbb{R}^{2}$ to the ball in Figure \ref{fig:relativeorientationaandb}.

The diagram $D$ is said to be \emph{coherently oriented} if for each pair of split components of $D$, say $D_1$ and $D_2$, there exists a sequence $D_1 = D_{i_1},...,D_{i_{k}}= D_2$ of split components of $D$ such that the components $D_{i_{j}}$ and $D_{i_{j+1}}$ have compatible orientations for each $j \in \{ 1, ..., k-1 \}$.

\begin{proposition}\label{proposition:uniquebeta}
Let $D$ be a coherently oriented link diagram. If $x(D)$ is invariant under coherent Reidemeister moves of the second type, then $x(D)$ is a ($R$-)multiple of either $ \beta_{\mathcal{F}}(D,R)$ or $\bar{\beta}_{\mathcal{F}}(D,R)$.
\end{proposition}
\begin{proof}
Let $D$ be a coherently oriented diagram and $D_{1}$,...,$D_{k}$ be its split components. By Lemma \ref{lemma:keylemmabetaunique} the labels of $x(D)$ on the components of the oriented resolution of a split component are exactly as in $\beta_{\mathcal{F}}(D,R)$ or as in $\bar{\beta}_{\mathcal{F}}(D,R)$, up to multiplication by an element of $R$.
By the definition of coherently oriented link diagram, given two split component, say $D_{i}$ and $D_j$, there exists a sequence $D_1 = D_{i_1},...,D_{i_{k}}= D_2$ of split components of $D$ such that the components $D_{i_{j}}$ and $D_{i_{j+1}}$ have compatible orientations for each $j \in \{ 1, ..., k-1 \}$.
By definition of compatible orientation it is possible to perform a second type coherent Reidemeister move using an arc of $D_{i_j}$ and and arc of $D_{i_{j+1}}$. Thus, again by Lemma \ref{lemma:keylemmabetaunique}, if the labels of the circles corresponding to $D_{i_j}$ in $x(D)$ are as in $\beta_{\mathcal{F}}(D,R)$ (up to multiplication by an element of $R$), then also the labels of the circles corresponding to $D_{i_{j+1}}$ in $x(D)$ are as in $\beta_{\mathcal{F}}(D,R)$. Similarly, if the the labels of the circles corresponding to $D_{i_j}$ in $x(D)$ are as in $\bar{\beta}_{\mathcal{F}}(D,R)$, then also the labels of the circles corresponding to $D_{i_{j+1}}$ in $x(D)$ are as in $\bar{\beta}_{\mathcal{F}}(D,R)$. So, if the label of a circle $\gamma$ in $x(D)$ is a $R$-multiple of $b_{\gamma}$ (resp. $\bar{b}_{\gamma}$), then $x(D)$ is an $R$-multiple of $\beta_{\mathcal{F}}(D,R)$ (resp. $\bar{\beta}_{\mathcal{F}}(D,R)$).  
\end{proof}

\begin{proof}[of Theorem \ref{theorem:main2}]
All the braid closures are coherently oriented link diagrams. Theorem \ref{theorem:main2} now follows straightforwardly from Proposition \ref{proposition:uniquebeta} .
\end{proof}

\section{Specialising to Bar-Natan theory}\label{sec:betaBN}

In this section we specialise our construction to the case $\beta = \beta_{BN}$. First we explore the relationship between the $\beta$-invariants, the Plamenevskaya invariant and the LNS-invariants. In the second part of this section we extract new transverse invariants from the homology classes of the $\beta$-invariants. Furthermore, we provide some sufficient condition for these new invariants to be non-effective, and prove that also the vanishing of the Plamenvskaya invariant is non-effective for small knots.

\subsection{Relationship with other invariants}

Recall that the base ring for the Frobenius algebra $BN$ is $\mathbb{F}[U]$. Moreover, the resulting Khovanov-type homology is bi-graded, and the multiplication by $U$ lowers the quantum degree by $2$. It follows immediately from the definitions that the $\beta$-invariants are bi-homogeneous cycles of bi-degree $(0,sl(B))$.

Let us start by pointing out the relationship between Bar-Natan, Khovanov and Lee theories. This relationship follows directly from the definitions and can be condensed into the following exact sequences
\begin{equation}
0 \to C_{BN}^{\bullet,\bullet + 2}(D) \overset{ U\cdot}{\longrightarrow} C_{BN}^{\bullet,\bullet}(D) \overset{ \pi_{Kh}}{\longrightarrow} C_{Kh}^{\bullet,\bullet}(D) \to 0
\label{Eq:BNKh}
\end{equation}
\begin{equation}
0 \to C_{BN}^{\bullet}(D) \overset{ (U-1)\cdot}{\longrightarrow} C_{BN}^{\bullet}(D) \overset{ \pi_{TLee}}{\longrightarrow} C_{TLee}^{\bullet}(D) \to 0
\label{Eq:BNTLee}
\end{equation}
Moreover, it is also immediate from the definition of Plameneskaya and LNS invariants that
\[\pi_{TLee}(\overline{\beta}(B)) = \psi^{-}(B),\quad \pi_{TLee}(\beta(B)) = \psi^{+}(B),\]
\[\pi_{Kh}(\beta(B)) = \pi_{Kh}(\overline{\beta}(B)) = \psi(B).\]
Since the LNS invariants are representatives of the canonical generators of twisted Lee theory (cf. \cite{Rasmussen10, Mackaayturnervaz07}), their homology classes are linearly independent over $\mathbb{F}$. The following proposition follows immediately
\begin{proposition}
Given an oriented link diagram $D$, then $[\beta(D)]$ and $[\overline{\beta}(D)]$ generate a rank $2$ $\mathbb{F}[U]$-submodule of $H_{BN}^{\bullet,\bullet}(D)$. In particular, $[\beta(D)]$ and $[\overline{\beta}(D)]$ are always non-trivial and non-torsion.
\end{proposition}
This is in stark contrast to the behaviour of the homology class of $\psi$. In fact, $[\psi(B)]$ tends to vanish quite easily (cf. \cite[Proposition 3]{Plamenevskaya06}). However, the vanishing of $[\psi]$ can be detected directly from $[\beta]$. More precisely, we have the following result.
 
\begin{proposition}\label{prop:psiebeta}
Given a braid $B$, the following conditions are equivalent
\begin{enumerate}
\item $[\psi(B)] = 0$;
\item exists $x\in H_{BN}^{0,sl(B)+2}(\widehat{B})$ such that $U x = [\beta(B)]$;
\item exists $x\in H_{BN}^{0,sl(B)+2}(\widehat{B})$ such that $U x = [\overline{\beta}(B)]$;
\end{enumerate}
\end{proposition}
\begin{proof}
Is immediate from the exact sequence in Equation \eqref{Eq:BNKh} and from the fact that $\pi_{Kh}(\beta) = \pi_{Kh}(\overline{\beta}) = \psi$.
\end{proof}

In order to prove the equivalence between the $\beta$- and the LNS-invariants we need a technical lemma. The proof of this lemma is quite easy and is left as an exercise.

\begin{lemma}[Unique homogeneous lift]\label{lemma:uniquelift}
Let $R$ be a PID, and let $M$ be a graded $R[U]$-module, where $deg(U) = -2$ and the graded structure on $R[U]$ is the natural one. Define a filtration on $M/(U-1)M$ as follows
\[\mathscr{F}_{i} = \left\langle [x] \in M/(U-1)M\: \vert\: x\in M_{j},\ j\leq i\right\rangle_{R}.\]
If $M$ is non-trivial only in either even or odd degree and if $[x] \in \mathscr{F}_{i}$, then there exists a unique $\tilde{x} \in M_{i}$ such that $[ \tilde{x}] = [x]$.
\end{lemma}

\begin{proof}[of Proposition \ref{prop:equivalence}]
Denote by $\ell$ the number of components of the Alexander closure of $B$. It is well known that Bar-Natan homology of $B$ and $B^\prime$ is supported in quantum degrees which are congruent to $\ell$ modulo $2$ (cf. \cite[Proposition 24]{Khovanov00}, or see \cite[Corollary 2.24]{Thesis1}).  Notice that $\beta$ (resp. $\overline{\beta}$) is the unique homogeneous lift of $\psi^{+}$ (resp. $\psi^{-}$) of quantum degree $sl$. Thus we can apply Lemma \ref{lemma:uniquelift}, and the statement follows.
\end{proof}

\begin{proposition}\label{prop:flype invariance}
The homology classes of the $\beta$-invariants are flype invariant in the following sense; if $B$ and $B^\prime$ are related by a negative flype, then there exists a (bi-graded) isomorphism of $\mathbb{F}[U]$-modules $F_*:H_{BN}^{\bullet,\bullet}(\widehat{B}) \to H_{BN}^{\bullet,\bullet}(\widehat{B^\prime})$ sending $[\beta(B)]$ and $[\overline{\beta}(B)]$ to $[\beta(B^\prime)]$ and $[\overline{\beta}(B^\prime)]$, respectively.
\end{proposition}
\begin{proof}
In \cite[Theorem 4.15]{Nglipsar13} it has been shown that there is a filtered chain homotopy equivalence $f:C_{TLee}^{\bullet}(\widehat{B}) \to C_{TLee}^{\bullet}(\widehat{B^\prime})$, such that \[ f(\psi^{\star}(B)) = \pm \psi^{\star}(B^\prime) + d_{TLee} x_\star\quad \text{for some }x_\star\in \mathscr{F}_{sl}C_{TLee}^{-1}(\widehat{B^\prime}),\]
where $\star \in \{ +,\: -\} $ and $sl = sl(B) = sl(B^\prime)$. Let $D$ be an oriented link diagram. Thanks to Lemma \ref{lemma:uniquelift} one can define, for each $j$, an isomorphism $\rho_{j}$ of $\mathbb{F}$-chain complexes, whose inverse is $\pi_{TLee}$, which fits into the following commutative square
\[ 
\xymatrix{
C_{BN}^{\bullet, j}(D) \bigoplus C_{BN}^{\bullet, j+1}(D)\ar[d]_{\footnotesize{\begin{pmatrix}
0 & U \\ 1 & 0
\end{pmatrix}}} \ar@/^/[rr]^{\pi_{TLee}} &&\ar@/^/[ll]^{\rho_{j}} \ar@{^(->}[d]\mathscr{F}_{j}C_{TLee}^{\bullet}(D) \\
C_{BN}^{\bullet, j-1}(D) \bigoplus C_{BN}^{\bullet, j}(D)\ar@/^/[rr]^{\pi_{TLee}} &&\ar@/^/[ll]^{\rho_{j}} \mathscr{F}_{j-1}C_{TLee}^{\bullet}(D) \\
}
\]
Recall that all the enhanced states in the Bar-Natan complex have quantum degree congruent modulo $2$ to the number of components of the link represented by $D$. It follows that one of the two summands on the right hand side of the square is always trivial. Now, we can use the $\rho_{j}$'s to define a lift $F$ of $f$ to Bar-Natan theory. Since $f$ is a filtered chain homotopy equivalence, the map $F$ will be graded, $\mathbb{F}[U]$-linear and a chain homotopy equivalence. Furthermore, Lemma \ref{lemma:uniquelift} ensures us that $F(\beta(B)) = \pm \beta (B^\prime) + d_{BN} \widetilde{x_+}$, and $F(\overline{\beta}(B)) = \pm \overline{\beta} (B^\prime) + d_{BN} \widetilde{x_-}$, where $\widetilde{x_\pm}$ is the unique lift of $x_\pm$ of quantum degree $sl$. Since the sign changes are coherent (\cite[Proof of Theorem 4.5]{Nglipsar13}), the claim follows.
\end{proof}

\begin{remark}
If the filtered degrees of $x_\pm$ were strictly lower than $sl$, we could have only proved that $U^k$ times the homology classes of the $\beta$-invariants were flype invariant, for some $k>0$.
\end{remark}

\begin{remark}
We have proved a slightly stronger statement; every time that the LNS-invariant are preserved as in Proposition \ref{prop:LipshitzNgSarkar}, then the homology classes of the $\beta$-invariants are invariant as in Proposition \ref{prop:flype invariance}.
\end{remark}

Proposition \ref{prop:flype invariance} implies that it is really difficult to use the homology classes of the $\beta$-invariants to distinguish flypes; even if one manages prove that the chain homotopy $F$ is not the chain homotopy equivalence associated to a sequence of Markov (or Reidemeister) moves, there will be the problem of how to extract information from the homology classes of the $\beta$-invariants. Moreover, the argument of Lipshitz, Ng and Sarkar can be modified to prove that the homology classes of the $\beta_{\mathcal{F}}$-invariants, for a large choice of $\mathcal{F}$, cannot be used to distinguish flypes.
\subsection{The $c$-invariants}
Let $B$ be a braid, and $\mathbb{F}$ be a field. Define the $c$\emph{-invariants} of $B$ (over $\mathbb{F}$) as follows
\[ c_{\mathbb{F}}(B) = max \left\lbrace k\: \vert \: [\beta(B)] = U^k x,\: \text{for some }x\in H_{BN}^{0,\bullet}(B, \mathbb{F}[U]) \right\rbrace\]
and
\[ \overline{c}_{\mathbb{F}}(B) = max \left\lbrace k\: \vert \: [\overline{\beta}(B)] = U^k x,\: \text{for some }x\in H_{BN}^{0,\bullet}(B, \mathbb{F}[U]) \right\rbrace\]
The $c$-invariants are, of course, transverse braid invariants. Moreover, they provide the same or less amount of transverse information as the homology classes of the $\beta$-invariants. In particular, we have the flype invariance.

\begin{proof}[Proposition \ref{proposition:main3}]
The proposition follows directly from Proposition \ref{prop:flype invariance}.
\end{proof}

Notice that $[\psi(B)] = 0$ if, and only if, $c_{\mathbb{F}}(B) >0$ (Proposition \ref{prop:psiebeta}). In particular, the $c$-invariants determine the vanishing of the homology class of $\psi$.

\begin{definition}
An oriented link $\lambda$ is called \emph{$c$-simple} if each non-simple pair of transverse representatives of $\lambda$ have the same $c$-invariants. 
\end{definition}

The non-effectiveness of the $c$-invariants is equivalent to all links being $c$-simple. We wish to address the following question: let $\lambda$ be an oriented link. What are the homological conditions which $\lambda$ should satisfy to be $c$-simple? 
 
This question is intentionally vague. For example, we did not specify which homology one should consider, or which type of condition one should look for. However, we manage to give some sufficient conditions for a knot to be $c$-simple.

First, we need to look more closely at the $\beta$-invariants. Let $B$ be a braid representing the knot $\kappa$. Denote by $s(\kappa) = s(\kappa,\mathbb{F})$ the Rasmussen invariant of $\kappa$ (\cite{Rasmussen10}). Fix an isomorphism of bi-graded $\mathbb{F}[U]$-modules
\begin{equation}\begin{small}
\phi : H_{BN}^{\bullet,\bullet}(\widehat{B},\mathbb{F}[U]) \to \bigoplus_{i=1}^{m} \frac{\mathbb{F}[U]}{(U^{t_{i}})} (h_{i},q_{i}) \oplus \mathbb{F}[U] (0,s(\kappa) +1)\oplus \mathbb{F}[U] (0,s(\kappa) -1),\end{small}
\label{fixiso2}
\end{equation}
which exists, for some choices of $h_{i}$, $q_{i}\in \mathbb{Z}$ and $t_{i} \in \mathbb{N}$, by the structure theorem for graded modules over a PID (see, for example, \cite[Theorem 3.19]{Zomorodian05}) and by \cite{Khovanov06}. Consider the natural generators of the module on the right hand side of \eqref{fixiso2}, that is
\[ e_{i} = (\overset{\overset{i-th\: \text{place}}{\downarrow}}{0,...,0,[1],0,...,0})\quad  f_{1} = (0,...,1,0)\quad \text{and}\quad f_{2} = (0,...,0,1),\]
where $i \in \{ 1,...,m\}$, and set
\[ \tilde{e}_i = \phi^{-1}(e_i) \quad \text{and}\quad \tilde{f}_{j} = \phi^{-1}(f_{j}).\]
Notice that for each $i$ we have
\[ (hdeg(\tilde{e}_i) , qdeg(\tilde{e}_i)) = (h_{i},q_{i}).\]
Denote by $I_{0}$ the set of all $i\in \{ 1,...,m\}$ such that $h_{i} = 0$.
From the definitions of the $\tilde{e}_{i}$'s, $\tilde{f}_1$, $\tilde{f}_{2}$ and $c_{\mathbb{F}}(B)$ it follows immediately that
\begin{equation}
[\beta(B,\mathbb{F})] = U^{c_{\mathbb{F}}(B)}\left(\alpha_1 U^{r_1} \tilde{f}_1 + \alpha_2U^{r_2} \tilde{f}_2 + \sum_{i\in I_{0}} \gamma_{i}U^{k_i}\tilde{e}_{i}\right),
\label{eq:homologyclassofbeta}
\end{equation}
where at least one among $r_1$, $r_2$, and the $k_i$'s (such that $\gamma_{i}U^{k_i}\tilde{e}_{i} \ne 0$) is zero. 
Moreover, as the homology classes of the $\beta$-invariants generate a rank $2$ $\mathbb{F}[U]$-sub-module of $H_{BN}^{\bullet,\bullet}(\kappa,\mathbb{F}[U])$, it follows that at least one among $\alpha_{1}$ and $\alpha_{2}$ is non-trivial.
Since the $\beta$-invariants are homogeneous, it follows that
\[ q_i -2 k_i =  s(\kappa) -1 - 2r_2 =  s(\kappa) + 1 - 2r_1 = sl(B) + 2 c_{\mathbb{F}}(B).\]
In particular, we get that
\[ r_1 = r_2 + 1.\]
If $r_1$ equals $0$ we obtain that
\[ s(\kappa) -1 = sl(B) +2c_{\mathbb{F}}(B).\]
Thus, under the above assumption $c_{\mathbb{F}}$ would be (half of) the difference between a knot invariant and the self linking, and hence non-effective. A similar reasoning applies to $\bar{c}_{\mathbb{F}}$. Making use of these considerations we can prove the following proposition.

\begin{proposition}\label{proposition:csimpleknotsBNs}
Let $\kappa$ be an oriented knot. If $q_i$ is greater than or equal to $s(\kappa) - 1$ for each $i\in I_{0}$, then $\kappa$ is $c$-simple, where $s(\kappa)$ denotes the Rasmussen invariant. Moreover, in this case we have
\[ s(\kappa) -1 = sl(B) +2c_{\mathbb{F}}(B),\]
for each braid representative $B$ of $\kappa$.
\end{proposition} 
\begin{proof}
If $q_i \geq s(\kappa) - 1$, then $k_{i}\geq r_2$. Thus, if $r_2 > 0$, then the $k_i$'s are also strictly greater than $0$. It follows that $r_1$ must be equal to $0$, and the claim follows.
\end{proof}

\begin{remark}
the proof of Proposition \ref{proposition:csimpleknotsBNs} works also if $\kappa$ is a link such that $H^{0,\bullet}_{BN}(\kappa,\mathbb{F}[U])/Tor (H_{BN}^{0,\bullet}(\kappa,\mathbb{F}[U]) )$ is supported in two quantum degrees. These links are called \emph{pseudo-thin} in \cite{Alberto15}.
\end{remark}

\begin{proof}[of Proposition \ref{corollary:csimpleknotsKh}]
Directly from the definitions, it follows that $C_{Kh}^{\bullet,\bullet}(D,\mathbb{F}) = C_{BN}^{\bullet,\bullet}(D,\mathbb{F}[U]) \otimes_{\mathbb{F}[U]} \mathbb{F}[U]/(U)$, for each oriented link diagram $D$. From the K\"unneth theorem follows immediately that if (1), (2) or (3) are satisfied, then $q_{i} \geq s(\kappa) - 1$ (see also \cite[Proposition 2.26]{Thesis1}). Proposition \ref{corollary:csimpleknotsKh} now follows from Proposition \ref{proposition:csimpleknotsBNs}.
\end{proof}

From the analysis of the Bar-Natan and Khovanov homologies of all prime knots with less than 12 crossings it follows that

\begin{corollary}\label{corollary:allknotswithlessthan12arecsimple}
Let $\mathbb{F}$ be a field such that $char(\mathbb{F}) \ne 2$. All prime knots with less than 12 crossings and their mirror images satisfy the conditions in Proposition \ref{corollary:csimpleknotsKh}. In particular, they are $c$-simple over $\mathbb{F}$.
\end{corollary}
\begin{proof}
For the computation of integral Khovanov homology the reader may refer to the KnotAtlas (\cite{KnotAtlas}). Since there is only $2$-torsion in the integral Khovanov homology of the prime knots with less than $12$ crossings, their Khovanov homology over $\mathbb{F}$ is concentrated in the same bi-degrees as their rational Khovanov homology. A well-known theorem due to Lee (\cite{Lee05}) states that alternating knots are $Kh$-thin. As a consequence of Proposition \ref{corollary:csimpleknotsKh}, all alternating knots are $c$-simple. So we may restrict our attention to the non-alternating knots.
According to KnotInfo (\cite{Knotinfo}), among the 249 prime knots with less than 11 crossings those which are non-alternating are the following
\[
\begin{matrix}
\textcolor{blue}{8_{19}} & \textcolor{red}{8_{20}} & \textcolor{red}{8_{21}} & 9_{42} & \textcolor{red}{9_{43}} & \textcolor{red}{9_{44}} & \textcolor{red}{9_{45}} & \textcolor{red}{9_{46}} & \textcolor{red}{9_{47}} & \textcolor{red}{9_{48}} \\
\textcolor{red}{9_{49}} & \textcolor{blue}{10_{124}} & \textcolor{red}{10_{125}} & \textcolor{red}{10_{126}} & \textcolor{red}{10_{127}} & \textcolor{blue}{10_{128}}  & \textcolor{red}{10_{129}}  & \textcolor{red}{10_{130}}  & \textcolor{red}{10_{131}} & \textcolor{blue}{10_{132}}  \\
\textcolor{red}{10_{133}} & \textcolor{red}{10_{134}} & \textcolor{red}{10_{135}} & 10_{136} & \textcolor{red}{10_{137}} & \textcolor{red}{10_{138}}  & \textcolor{blue}{10_{139}}  & \textcolor{red}{10_{140}}  & \textcolor{red}{10_{141}} & \textcolor{red}{10_{142}}  \\
\textcolor{red}{10_{143}} & \textcolor{red}{10_{144}} & \textcolor{blue}{10_{145}} & \textcolor{red}{10_{146}} & \textcolor{red}{10_{147}} & \textcolor{red}{10_{148}}  & \textcolor{red}{10_{149}}  & \textcolor{red}{10_{150}}  & \textcolor{red}{10_{151}} & \textcolor{blue}{10_{152}}  \\
\textcolor{blue}{10_{153}} & \textcolor{blue}{10_{154}} & \textcolor{red}{10_{155}} & \textcolor{red}{10_{156}} & \textcolor{red}{10_{157}} & \textcolor{red}{10_{158}}  & \textcolor{red}{10_{159}}  & \textcolor{red}{10_{160}}  & \textcolor{blue}{10_{161}} & \textcolor{red}{10_{162}}  \\
\textcolor{red}{10_{163}} & \textcolor{red}{10_{164}} & \textcolor{red}{10_{165}} &  &   &   &    &    &  &  \\
\end{matrix}
\]
The ones marked in red are the $Kh$-thin knots, while those in blue are the $Kh$-pseudo-thin knots. If a knot is $Kh$-thin or $Kh$-pseudo-thin, then also its mirror image is $Kh$-thin or $Kh$-pseudo-thin. Thus, by Proposition \ref{corollary:csimpleknotsKh}, all prime knots in the list above (and also their mirrors) are $c$-simple, except $9_{42}$ and $10_{136}$. These knots satisfy condition (2) of Proposition \ref{corollary:csimpleknotsKh} and hence they are $c$-simple.

Finally, among the non-alternating prime knots with 11 crossings and their mirrors the ones which are neither pseudo-thin nor satisfy the condition (2) of Proposition \ref{corollary:csimpleknotsKh} are
\[
\begin{matrix}
m(11_{n12}) & m(11_{n24}) & 11_{n34} & m(11_{n34})  & 11_{n42} \\
m(11_{n42})  &   m(11_{n70})  &   m(11_{n79}) & 11_{n92} & m(11_{n96}) \\
\end{matrix}
\]
It is know that if $char(\mathbb{F})\ne 2$ the torsion sub-module of $H^{\bullet,\bullet}_{BN}(\kappa,\mathbb{F}[U])$ is isomorphic to the $\mathbb{F}[U]$-module
\[M = \bigoplus_{i=1}^{m} \frac{\mathbb{F}[U]}{(U^{2k_{i}})},\]
for some $m$, $k_{1}$, ..., $k_{m}\in\mathbb{N}\setminus \{ 0 \}$ (\cite[Corollary 2.33]{Thesis1}). The links listed above satisfy point (3) of Proposition \ref{corollary:csimpleknotsKh}. Hence they are $c$-simple and the claim follows. 
\end{proof}


The reader should take into account that knots with less than 13 crossings seem to have pretty ``simple'' Khovanov homology. For example, the first prime knot to have Khovanov homology supported in more than three diagonals, which is also the first with thick torsion, is the knot $13n3663$ (see \cite[Appendix A.4]{Shumakovitch04}). Unfortunately, there is a paucity of examples of non-simple pairs (not obtained via flype) whose underlying knot has crossing number $13$ or $14$, so it is difficult to computationally explore the effectiveness of the $c$-invariants (and of $\psi$).

\bibliography{Bibliography.bib}

\begin{thebibliography}{10}

\bibitem{BarNatan05cob}
D.~Bar-Natan.
\newblock Khovanov homology for tangles and cobordisms.
\newblock {\em Geometry \& Topology}, 9:1443--1499, 2005.

\bibitem{KnotAtlas}
D.~Bar-Natan and S.~Morrison et~al.
\newblock The {K}not {A}tlas.
\newblock http://katlas.org/wiki/.

\bibitem{Bennequin83}
D.~Bennequin.
\newblock Entrelacements et équations de \text{P}faff.
\newblock {\em Astérisque}, 107--108, 1983.

\bibitem{BirmanMenasco06II}
J.~Birman and W.~Menasco.
\newblock Stabilization in the braid groups \text{II}: Transversal simplicity
  of knot.
\newblock {\em Geometry \& Topology}, 10:1425--1452, 2006.

\bibitem{Alberto15}
A.~Cavallo.
\newblock On the slice genus and some concordance invariants for links.
\newblock {\em Journal of knot theory and its ramifications}, 24(4), 2015.

\bibitem{Knotinfo}
J.~Cha and C.~Livingston.
\newblock Knot{I}nfo.
\newblock \small{http://www.indiana.edu/{~}knotinfo/}.

\bibitem{Thesis1}
C.~Collari.
\newblock {\em Transverse invariants from the deformations of Khovanov
  $\mathfrak{sl}_{2}$- and $\mathfrak{sl}_3$-homologies}.
\newblock PhD thesis, Universit\`{a} di Firenze, 2016.

\bibitem{Etnyre05}
J.~B. Etnyre.
\newblock Legendrian and transversal knots.
\newblock {\em Handbook of Knot Theory}, pages 105--186, 2005.

\bibitem{EtnyreHonda03}
J.~B. Etnyre and K.~Honda.
\newblock Cabling and transverse simplicity.
\newblock {\em Annals of Mathematics. Second Series}, 162(3):1305--1333, 2005.

\bibitem{Jacobsson04}
M.~Jacobsson.
\newblock An invariant of link cobordisms from khovanov homology.
\newblock {\em Algebraic \& Geometric Topology}, 4:1211--1251, 2004.

\bibitem{Kadison99}
L.~Kadison.
\newblock {\em New Examples of \text{F}robenius Extensions}, volume~14 of {\em
  University Lecture Series}.
\newblock American Mathematical Society, Providence, 1999.

\bibitem{Khovanov00}
M.~Khovanov.
\newblock A categorification of the \text{J}ones polynomial.
\newblock {\em Duke Mathematical Jorunal}, 101:359--426, 2000.

\bibitem{Khovanov06}
M.~Khovanov.
\newblock Link homology and \text{F}robenius extensions.
\newblock {\em Fundationes Matematic\text{\ae}}, 190:179--190, 2006.

\bibitem{Krasner09}
D.~Krasner.
\newblock A computation in {K}hovanov-{R}ozansky homology.
\newblock {\em Fundamenta Mathematicae}, 203(1):75--95, 2009.

\bibitem{Lee05}
E.~S. Lee.
\newblock An endomorphism of the \text{K}hovanov invariant.
\newblock {\em Advances Mathematic\text{\ae}}, 197(2):554--586, 2005.

\bibitem{Nglipsar13}
R.~Lipshitz, L.~Ng, and S.~Sarkar.
\newblock On transverse invariants from \text{K}hovanov homology.
\newblock {\em Quantum Topology}, 6(3):475--513, 2015.

\bibitem{Mackaayturnervaz07}
M.~Mackaay, P.~Turner, and P.~Vaz.
\newblock A remark on \text{R}asmussen's invariant of knots.
\newblock {\em Journal of knot theory and its ramifications}, 16(3):333--344,
  2007.

\bibitem{NgOzsvathThurston08}
L.~Ng, P.~Ozsv\'{a}th, and D.~Thurston.
\newblock Transverse knots distinguished by knot {F}loer homology.
\newblock {\em Journal of Symplectic Geometry}, 6(14):461--490, 2008.

\bibitem{OrevkovShev03}
S.~Orevkov and V.~Shevchishin.
\newblock Markov theorem for transverse links.
\newblock {\em Journal of knot theory and its ramifications}, 12(7):905--913,
  2003.

\bibitem{Plamenevskaya06}
O.~Plamenevskaya.
\newblock Transverse knots and \text{K}hovanov homology.
\newblock {\em Mathematical Research Letters}, 13(4):571--586, 2006.

\bibitem{Plamenevskaya15}
O.~Plamenvskaya.
\newblock Transverse invariants and right-veering.
\newblock {\em Preprint}, 2015.
\newblock Available on Arxiv:1509.01732.

\bibitem{Polyak10}
M.~Polyak.
\newblock Minimal set generating reidemeister moves.
\newblock {\em Quantum Topology}, 1(4):399--411, 2010.

\bibitem{Rasmussen10}
J.~Rasmussen.
\newblock Khovanov homology and the slice genus.
\newblock {\em Inventiones Mathematic\text{\ae}}, 182(2):419--447, 2010.

\bibitem{Shumakovitch04}
A.~Shumakovitch.
\newblock Torsion of the \text{K}hovanov homology.
\newblock {\em Available on ArXiv}, 2004.
\newblock http://arXiv.org/abs/math/0405474v1.

\bibitem{Turner04}
P.~Turner.
\newblock Calculating \text{B}ar-\text{N}atan's characteristic two
  \text{K}hovanov homology.
\newblock {\em Journal of knot theory and its ramifications}, 68:1335--1356,
  2006.

\bibitem{Turner06}
P.~Turner.
\newblock Five lectures on \text{K}hovanov homology.
\newblock {\em Available on ArXiv}, 2006.
\newblock http://arXiv.org/abs/math/0606464v1.

\bibitem{Wrinkle03}
N.~Wrinkle.
\newblock The \text{M}arkov theorem for transverse knots.
\newblock {\em Available on ArXiv}, 2002.
\newblock http:\\www.arxiv.org/abs/math.GT/0202055.

\bibitem{Zomorodian05}
A.~J. Zomorodian.
\newblock {\em Topology for Computing}, volume~16 of {\em Cambridge Monographs
  on Applied and Computational Mathematics}.
\newblock Cambridge University Press, 2005.

\end{thebibliography}
\bibliographystyle{plain}

\end{document}